\documentclass[article,11pt,reqno]{amsart}
\usepackage[top=25mm,bottom=25mm,left=25mm,right=25mm]{geometry}
\usepackage{color}
\usepackage{amsmath}
\usepackage{graphicx}
\usepackage{tikz}
\usetikzlibrary{arrows,decorations.markings}

\usepackage{mathrsfs}
\usepackage{calligra}
\usepackage[mathscr]{eucal}

\newtheorem{lemma}{Lemma}[section]
\newtheorem{proposition}{Proposition}[section]
\newtheorem{theorem}{Theorem}[section]
\newtheorem{corollary}{Corollary}[section]
\newtheorem{definition}{Definition}[section]

\newtheorem{remark}{Remark}[section]
\newtheorem{assumption}{Assumption}[section]

\makeatletter
\def\section{\@startsection{section}{1}%
\z@{1\linespacing\@plus\linespacing}{1\linespacing}%
{\bf\centering}}
\def\subsection{\@startsection{subsection}{0}%
\z@{\linespacing\@plus\linespacing}{\linespacing}%
{\bf}}
\makeatother

\makeatletter
\@addtoreset{equation}{section}
\makeatother

\DeclareMathOperator{\supp}{supp}

\DeclareMathOperator{\Spec}{Spec}

\newcommand{\Norm}[2]{\left\Vert #1 \right\Vert_{#2}}

\providecommand{\seq}[1]{(#1_k)_{k\in \mathbb{N}}}
\providecommand{\sequ}[1]{(#1_n)_{n\in \mathbb{N}}}

\newcommand{\cS}{\mathcal{S}}

\newcommand{\cD}{\mathcal{D}}
\newcommand{\cA}{\mathcal{A}}
\newcommand{\cL}{\mathcal{L}}

\newcommand{\cK}{\mathcal{K}}

\newcommand{\cF}{\mathcal{F}}
\newcommand{\cE}{\mathcal{E}}
\newcommand{\cB}{\mathcal{B}}

\newcommand{\cH}{\mathcal{H}}

\newcommand{\cU}{\mathcal{U}}
\newcommand{\R}{\mathbb{R}}

\newcommand{\N}{\mathbb{N}}
\definecolor{mr}{rgb}{0.1,0.2,0.7}

\newcommand{\vertiii}[1]{{\left\vert\kern-0.25ex\left\vert\kern-0.25ex\left\vert #1
		\right\vert\kern-0.25ex\right\vert\kern-0.25ex\right\vert}}

\begin{document}
\title[Stability of eigenvalues and monotonicity]
{\small Stability of ground state eigenvalues of non-local Schr\"odinger operators with respect to potentials
and applications}
\author{Giacomo Ascione and J\'ozsef L{\H o}rinczi}
\address{Giacomo Ascione,
Scuola Superiore Meridionale \\
80138 Napoli, Italy}
\email{giacomo.ascione@unina.it}
\address{J\'ozsef L\H orinczi,
Alfr\'ed R\'enyi Institute of Mathematics \\
1053 Budapest, Hungary}
\email{lorinczi@renyi.hu}

\begin{abstract}
In a first part of this paper we investigate the continuity (stability) of the spectrum
of a class of non-local Schr\"odinger operators on varying the potentials. By imposing conditions of different
strength on the convergence of the sequence of potentials, we give either direct proofs to show the strong or
norm resolvent convergence of the so-obtained sequence of non-local Schr\"odinger operators, or via
$\Gamma$-convergence of the related positive forms for more rough potentials. In a second part we use these
results to show via a sequence of suitably constructed approximants that the ground states of massive or
massless relativistic Schr\"odinger operators with spherical potential wells are radially decreasing functions.


\bigskip
\noindent
2020 {\it MS Classification}: Primary 47D08, 47A75; Secondary 47G20, 35R11
\end{abstract}

\maketitle

\baselineskip 0.55 cm
\tableofcontents

\newpage

\section{Introduction}
The basic motivation of this paper was to prove that, roughly, for a relativistic Schr\"odinger operator of the type
$(-\Delta + m^2)^{1/2}-m-v\mathbf 1_{B_a}$, where $m \geq 0$ is the rest mass of the particle, $B_a$ is a $d$-dimensional
ball of radius $a$ centred in the origin, and $v > 0$ is the depth of the spherical potential well (a coupling constant),
the ground state of this Hamiltonian is a radially monotone decreasing function in space, whenever it exists. Such a
behaviour is reasonable to expect but it turns out to require some work to show it rigorously, which is presented in this
paper for a wider class of non-local Schr\"odinger operators. An immediate implication of this
monotonicity is that precise \emph{local} (i.e., non-asymptotic) estimates can be derived on the ground state, as presented
in \cite{AL2}, which do not otherwise seem to come about by alternative methods currently available. Our approach is applicable
to a range of further problems regarding the local behaviour of solutions of non-local equations.

To carry this proof through we need to vary the potentials, which requires a control on convergence, i.e., a study of the
question whether for a given non-local operator $L$ and a sequence $\sequ V$ of potentials, for which the corresponding
non-local Schr\"odinger operators $\cH_n = L + V_n$ are assumed to have the ground state eigenvalues $\lambda_n$ and eigenfunctions
$\varphi_n$, these objects converge to the eigendata of $\cH = \lim_{n\to\infty} \cH_n$, when the limit is taken in an
appropriate sense. While there are classic results using strong/norm resolvent convergence due to Kato \cite[Ch.8, Sect.3]{K80}
and Weidmann \cite{W80} guaranteeing this stability (or continuity) of the eigenvalues for a broad class of self-adjoint operators
under conditions of varying strength, our approach here also aims to develop a framework based on $\Gamma$-convergence of positive
quadratic forms. This has several advantages. One is that the conditions we draw are rather mild and can be expressed directly
in explicit terms of the sequence of potentials, and by a verifiable criterion $\Gamma$-convergence implies the more conventional
strong resolvent convergence. Secondly, our framework also accommodates ``degenerate" cases when the sequence of potentials
supported on full space $\R^d$ replicates the Dirichlet exterior value problem in the limit, corresponding to the situation of
having zero potential in a bounded domain of $\R^d$ and infinite potential outside. Thirdly, a version of our approach also
covers the cases when instead of varying the potentials, the kinetic terms are varied, i.e., then continuity of the eigenvalues
of a sequence $\cH_n = L_n + V$ is considered, where, for instance, $L_n = (-\Delta)^{s_n}$, and $\sequ s \subset (0,2)$ is some
sequence convergent to one of the endpoints of the interval, which will be explored elsewhere.

Tools relying on $\Gamma$-convergence proved to be rather powerful in various contexts in the literature. For instance, it has
been used to show stability of variational eigenvalues for the fractional $p$-Laplacian by Brasco, Parini and Squassina \cite{BPS16}.
Dell'Antonio \cite{DA19,DA21} studied the convergence of regularized three-body Hamiltonians to contact (or point) interactions,
important in the understanding of the Efimov effect. Chen and Song used the related Mosco-convergence \cite{M93} of eigenvalues
of generators of subordinate L\'evy processes in domains \cite{CS06}, Song and Li to regular subspaces of one-dimensional
diffusion processes \cite{SL16}, and Li, Uemura and Ying similarly to regular Dirichlet subspaces \cite{LUY}. Ponce provided a
$\Gamma$-convergence based proof of the convergence of integrals initially studied by Br\'ezis, Bourgain and Mironescu \cite{P04}.

The main objects of this paper are non-local Schr\"odinger operators of the form $\cH= \Phi(-\Delta) + V$, where the kinetic term
is given by a Bernstein function $\Phi$ of the Laplacian $\Delta$, and the potential $V$ is a multiplication operator picked from
a Kato-class defined with respect to an integrability condition relating with $\Phi$. (For the details see Sections 2.1-2.3 below.)
Such non-local
Schr\"odinger operators have been first introduced and analyzed in \cite{HIL12}, covering massless (fractional) and massive
relativistic Schr\"odinger operators as well as others whose jump kernels may have a heavy or a light decay at infinity, and they
have been the object of extensive study recently.
Bernstein functions of the Laplacian and related
non-local equations have been used since also in other directions such as a generalization of the Caffarelli-Silvestre extension
technique \cite{KM}, maximum principles \cite{BL19,BL21}, the blow-up of solutions of stochastic PDE with white or coloured noise
\cite{DLN}, or the theory of embedded eigenvalues and scattering \cite{IW20,ILS}.

In this set-up, our work here has two main parts. In a first part we describe these operators and their related Dirichlet forms
and form domains, introducing a space $H^\Phi$ as a counterpart of the fractional Sobolev space $H^s$, $s \in (0,1)$, for more
general Bernstein
functions $\Phi$ (Section 2), and a corresponding semi-norm $[u]_{\Phi}$ generalizing the Gagliardo semi-norm. We will establish
that $H^\Phi$ is continuously embedded in $H^s$ (Proposition \ref{prop:embed}), allowing controls in terms of the different norms.
Also, we will define Dirichlet forms associated with
$\Phi(-\Delta)$ and with their perturbations by $\Phi$-Kato class potentials. Then, taking a sequence $\sequ V$ of suitably chosen
potentials convergent to a potential $V$, we establish the strong resolvent convergence of $\cH_n = \Phi(-\Delta) + V_n$, $n \in
\mathbb N$, to $\cH = \Phi(-\Delta) + V$ via the $\Gamma$-convergence of related positive quadratic forms (Theorem \ref{thm:stability}).
This will, in particular, yield the stability of the spectra in the sense discussed above (Corollary \ref{cor:stabilityspectrum}).
Under two further sets of conditions we obtain strong resolvent convergence (Theorems \ref{thm:stability3}) and norm resolvent
convergence (Theorem \ref{thm:stability4}) of $\sequ \cH$ directly, leading to the same conclusion.

In a second part (Section 4) we then turn to apply these stability results to the special case of the family
$$
\cH = (-\Delta + m^{2/\alpha})^{\alpha/2} - m - v\mathbf  1_{B_a}, \quad  0 < \alpha < 2, \; m \geq 0, \; a, v > 0,
$$
of fractional relativistic operators perturbed by spherical potential wells. By using appropriate mollifiers from the inside to
the outside of the potential well, first we construct approximants of the ground state of $\cH$ (Section 4.1). We show that they
are bounded smooth functions (Proposition \ref{prop:regular}) and replicate the ground state $\varphi$ of $\cH$ in the limit
(Proposition \ref{appx}). Making use of the sequence of approximants we then conclude that their images under the kinetic term
$L_{m,\alpha} = (-\Delta + m^{2/\alpha})^{\alpha/2} - m$ of the operator converge in $L^2$ sense to $L_{m,\alpha}\varphi$
(Theorem \ref{thm:regularphi0}). Showing that the ground state $\varphi$ is rotationally symmetric (Proposition \ref{rotsym}),
we then prove that it is radially non-increasing (Theorems \ref{prop:mono2}-\ref{monout}). A crucial technical step in order
to achieve this relies on an estimate on $L_{m,\alpha}w$ at its minimum, where $w$ is $C^2$, bounded, and antisymmetric with
respect to a hyperplane (Lemma \ref{lem:rellapcont}), relating with a maximum principle for narrow domains.

\section{Preliminaries}
\subsection{Bernstein functions and Gagliardo-type semi-norms}
We start by a word on the notation. We write $\Spec_{\rm ess}(A)$ for the essential spectrum, $\Spec_{\rm d}(A)$ for
the discrete spectrum, and $\Spec(A)$ for the full spectrum of operator $A$. The Schwartz space of complex-valued
functions on $\R^d$ will be denoted by $\cS(\R^d;\mathbb C)$, and we simply write $\cS(\R^d)$ for the similar object
for real-valued functions. Also, we denote the space of square-integrable complex-valued functions on $\R^d$ by
$L^2(\R^d;\mathbb{C})$.
For the Fourier transform of a function $h$ we write $\widehat{h}:=\cF[h]$ interchangeably, as
simpler or convenient. Scalar product in $\R^d$ will be denoted by pointed brackets $\langle \cdot, \cdot \rangle$.
For $L^p$ norms, we write for simplicity $\Norm{u}{p} = \Norm{u}{L^{p}(\R^d)}$, while spell out
the subscript in detail when we use a set other than $\R^d$; below we will work with several different norms but no
confusion will occur caused by this shorthand. A constant $C$ dependent on parameters $a, b, ...$ will be denoted by
$C(a,b,...)$. A ball of radius $r > 0$ centered in $x \in \R^d$ will be denoted by $B_r(x)$ and simply by $B_r$ when
$x=0$.

Recall that a Bernstein function is an infinitely differentiable non-negative function on the positive semi-axis,
whose derivative is completely monotone, i.e., an element of the convex cone
$$
\mathcal B = \left\{\Phi \in C^\infty((0,\infty)): \, \Phi \geq 0 \;\; \mbox{and} \:\; (-1)^{n}\frac{d^n \Phi}
{dx^n} \leq 0, \; \mbox{for all $n \in \mathbb N$}\right\}.
$$
In particular, Bernstein functions are increasing and concave. It is well-known that Bernstein functions have
the following canonical (L\'evy-Khintchine) representation: 
A function $\Phi:(0,\infty) \to \R$ belongs to $\cB$ if and only if there exist $a_\Phi,b_\Phi \ge 0$ and a measure
$\mu_\Phi$ on the positive semi-axis satisfying $\int_{(0,\infty)}(1\wedge t)\mu_\Phi(dt)<\infty$ such that for every
$z>0$ the expression
$$
\Phi(z)=a_\Phi+b_\Phi z+\int_{0}^{\infty}(1-e^{-zt})\mu_\Phi(dt)
$$	
holds. The triplet $(a_\Phi,b_\Phi,\mu_\Phi)$ determines $\Phi$ uniquely and vice versa. In particular, the measure
$\mu_\Phi$ is a L\'evy measure as it also satisfies $\int_{\R}(1\wedge t^2)\mu_\Phi(dt)<\infty$. Furthermore, a
Bernstein function $\Phi$ is said to be complete if its L\'evy measure $\mu_\Phi(ds)=m_\Phi(s)ds$ is such that
$m_\Phi$ is a completely monotone function. The set $\mathcal B_{\rm C} \subset \cB$ of complete Bernstein functions is
again a convex cone. Below we will restrict to the subset
$$
{\mathcal B}_0 = \big\{\Phi \in \mathcal B_{\rm C}: \, a_\Phi=b_\Phi=0 \big\}.
$$
A standard reference on Bernstein functions is \cite{SSV}, in which many details and examples are presented.

Fix $d \in \N$,
let $\Phi \in \cB$, and define the space
\begin{equation*}
H^{\Phi}(\R^d):= \Big\{h \in L^2(\R^d): \ \sqrt{\Phi(|\cdot|^2)} \; \widehat{h} \in L^2(\R^d)\Big\}.
\end{equation*}
Furthermore, for any $\Phi \in \mathcal B_{\rm C}$ we define the jump kernel
\begin{equation}
\label{jumpdens}
j_\Phi(r):=\int_0^{\infty}\frac{1}{\sqrt{2\pi t}}e^{-\frac{r^2}{4t}}m_\Phi(t)dt.
\end{equation}
We recall some properties of the function $j_\Phi$, for a proof see \cite[Prop. 2.2]{AL}.
\begin{lemma}
\label{lem:propj}
The jump kernel $j_\Phi$ satisfies the following properties:
\begin{itemize}
\item[(1)]
$j_\Phi(r)<\infty$ for every $r>0$;
\item[(2)]
$j_\Phi$ is a non-negative continuous and non-increasing function with $\lim_{r \to \infty}j_\Phi(r)=0$;
\item[(3)]
the integrability relations
$\int_{0}^{1}r^{d+1}j_\Phi(r)dr< \infty$
and $\int_{1}^{\infty}r^{d-1}j_\Phi(r)dr<\infty$
hold.
\end{itemize}
\end{lemma}
Next we characterize the space $H^{\Phi}(\R^d)$ by reformulating an earlier result that first appeared in
\cite[Lem. 3.1]{JS05}, stated for Schwartz space functions.
\begin{lemma}
\label{lem:GagliardotoFourier}
Let $u:\R^d \to \R$ be a measurable function. Then $u \in H^{\Phi}(\R^d)$ if and only if $u \in L^2(\R^d)$
and
\begin{equation*}
[u]_{\Phi}^2:=\frac{1}{2}\int_{\R^d}\int_{\R^d}|u(x+h)-u(x)|^2j_{\Phi}(|h|)dx dh<\infty.
\end{equation*}
\end{lemma}
\begin{proof}
Let $\psi(x)=1/(4x)$. We determine a Radon measure ${\nu}_\Phi$ on $(0,\infty)$ such that
\begin{equation*}
\mu_\Phi(ds)=(4\pi s)^{d/2}\psi({\nu}_\Phi)(ds),
\end{equation*}
where $\psi({\nu}_\Phi)(A)={\nu}_\Phi(\psi^{-1}(A))$ for every Borel set $A \subset (0,\infty)$. Assume that ${\nu}_\Phi$
is absolutely continuous with respect to Lebesgue measure and keep denoting its density still by ${\nu}_\Phi$. Setting
$A=[t_1,t_2]$ with $0<t_1<t_2$, we have
\begin{equation*}
\psi({\nu}_\Phi)(A)=\int_{\frac{1}{4t_2}}^{\frac{1}{4t_1}}{\nu}_\Phi(s)ds=\int_{t_1}^{t_2}{\nu}_\Phi\left(\frac{1}{4s}\right)
\frac{1}{4s^2}ds.
\end{equation*}
Hence we get
\begin{equation*}
m_\Phi(s)=(4\pi s)^{d/2}{\nu}_\Phi\left(\frac{1}{4s}\right)\frac{1}{4s^2}
\end{equation*}
and thus
\begin{equation*}
{\nu}_\Phi(s)=\frac{s^{\frac{d}{2}-2}}{4\pi^{d/2}}m_\Phi\left(\frac{1}{4s}\right).
\end{equation*}
Computing its Laplace transform we obtain
\begin{align*}
\int_{0}^{\infty}e^{-rs}{\nu}_\Phi(s)ds&
=
\int_{0}^{\infty}e^{-rs}s^{\frac{d}{2}-2}\pi^{-\frac{d}{2}}4^{-1}m_\Phi\left(\frac{1}{4s}\right)ds\\
&=
\int_0^{\infty}\frac{1}{\sqrt{2\pi t}}e^{-\frac{r}{4t}}m_\Phi(t)dt=j_\Phi(\sqrt{r}).
\end{align*}
By a combination with \cite[Lem. 2.1, Lem. 3.1]{JS05} the result follows.
\end{proof}

In the following we will need to compare the semi-norm $[u]_{\Phi}$ with the standard Gagliardo semi-norm
of order $s \in (0,1)$, defined by
\begin{equation*}
	[[u]]_{s}^2:=\int_{\R^d}\int_{\R^d}\frac{|u(x+y)-u(x)|^2}{|y|^{d+2s}}dydx,
\end{equation*}
giving rise to the fractional Sobolev space (for a useful survey see \cite{DPV12})
$$
H^s(\R^d):= \big\{u \in L^2(\R^d): \ [[u]]_s^2< \infty\big\}.
$$
The space $H^s(\R^d)$ coincides with $H^\Phi(\R^d)$ for $\Phi(z)=z^s$, $s \in (0,1)$. For this case
Lemma \ref{lem:GagliardotoFourier} has been shown to hold in an even more general anisotropic setting
\cite{MN78}.

Clearly, $[\cdot]_\Phi$ is just a semi-norm on $H^\Phi(\R^d)$ and we can define a norm by
\begin{equation*}
\Norm{u}{\Phi}:=\Norm{u}{2}+[u]_\Phi.
\end{equation*}
Similarly, an equivalent norm on $H^s(\R^d)$ is given by
\begin{equation*}
\vertiii{u}_{s}:=\Norm{u}{2}+[[u]]_s.
\end{equation*}
In regard to this norm, we will need the following Sobolev-type inequalities (see, for instance,
\cite[Ths. 6.7, 6.10, 8.2]{DPV12} for a more general formulation).
\begin{proposition}
\label{thm:sobolevt}
The following properties hold:
\begin{itemize}
\item[(1)]
Let $d \ge 2$, or $d=1$ and $s<\frac{1}{2}$. There exists a constant $C=C_{d,s}>0$ such that
\begin{equation*}
\Norm{u}{p^\ast} \le C_{d,s}\vertiii{u}_{s},
\end{equation*}
for every $u \in H^s(\R^d)$, where $p^\ast=p^\ast(d,s)=\frac{2d}{d-2s}$.

\item[(2)]
Let $d=1$ and $s=\frac{1}{2}$. There exists a constant $C>0$ such that
\begin{equation*}
\Norm{u}{q} \le C\vertiii{u}_{s},
\end{equation*}
for every $u \in H^s(\R^d)$ and every $2 \le q<\infty$.

\item[(3)]
Let $d=1$ and $s>\frac{1}{2}$. There exists a constant $C=C_{s}>0$ such that
\begin{equation*}
\Norm{u}{C^{0,\eta}(\R^d)} \le C_{s}\vertiii{u}_{s},
\end{equation*}
for every $u \in H^s(\R^d)$, where $\eta=s-\frac{1}{2}$.
\end{itemize}
\end{proposition}

In this paper we will work under the following standing assumptions.
\begin{assumption}
\label{assPGc}
Let $\Phi \in \mathcal B_0$ and $j_\Phi$ be the jump kernel associated with $\Phi$ as given in \eqref{jumpdens}.
The following hold:
\begin{enumerate}
\item
For every $t>0$ we require
$\int_{\R^d}e^{-t\Phi(|\xi|^2)}d\xi<\infty$.
\item
We assume that there exist constants $C>0$ and $s \in (0,1)$ such that
\begin{equation*}
j_\Phi(r)\ge \frac{C}{r^{d+2s}}, \quad r \in (0,1].
\end{equation*}
\end{enumerate}
\end{assumption}
\noindent
Condition (1) has also been used in \cite[Assmp. 4.1]{HIL12}. Condition (2) is satisfied for any $\Phi$ regularly
varying at infinity, while the slowly varying $\Phi(z) = \log(1 + z)$ does not, which can be checked by the jump
kernel explicitly known in this case.

These conditions allow us to compare the spaces $H^s(\R^d)$ and $H^\Phi(\R^d)$ as laid out in the following statement.
\begin{proposition}
\label{prop:embed}
Let $\Phi \in \mathcal{B}_0$ satisfy Assumption \ref{assPGc}. There exists a constant $C=C(d,s)>0$ such that
for every $u \in H^{\Phi}(\R^d)$ we have
$\vertiii{u}_{s}\le C\Norm{u}{\Phi}$,
i.e., $H^{\Phi}(\R^d)$ is continuously embedded in $H^s(\R^d)$.
\end{proposition}
\begin{proof}
Clearly, we only need to prove that $[[u]]_s^2 \le C\big([u]_\Phi^2+\Norm{u}{2}^2\big)$ for all $u \in H^\Phi(\R^d)$.
By using Assumption \ref{assPGc}(2),
\begin{align*}
[u]_\Phi^2&\ge \int_{\R^d}\int_{|y|\le 1}|u(x)-u(x+y)|^2j_\Phi(|y|)dydx\\
&\ge\int_{\R^d}\int_{|y|\le 1}|u(x)-u(x+y)|^2|y|^{-d-2s}dydx.
\end{align*}
On the other hand,
\begin{equation*}
\int_{\R^d}\int_{|y|\le 1}|u(x)-u(x+y)|^2|y|^{-d-2s}dy=[[u]]_s^2-\int_{\R^d}\int_{|y|\ge 1}|u(x)-u(x+y)|^2|y|^{-d-2s}dydx,
\end{equation*}
so that
\begin{align*}
[[u]]_s^2\le [u]_\Phi^2+\int_{\R^d}\int_{|y|\ge 1}|u(x)-u(x+y)|^2|y|^{-d-2s}dydx.
\end{align*}
Furthermore, we have
\begin{align*}
\int_{\R^d}&\int_{|y|\ge 1}|u(x)-u(x+y)|^2|y|^{-d-2s}dy=\int_{\R^d}\int_{|x-y|\ge 1}|u(x)-u(y)|^2|x-y|^{-d-2s}dydx\\
&\le
2\int_{\R^d}\int_{|x-y|\ge 1}(|u(x)|^2+|u(y)|^2)|x-y|^{-d-2s}dydx \\
&=
4\Norm{u}{2}^2\int_{|y| \ge 1}|y|^{-d-2s}dy =
\frac{4\sigma_{d}}{2s}\Norm{u}{2}^2,
\end{align*}
where $\sigma_d$ is the surface area of the $d$-dimensional unit sphere.
\end{proof}
We will also make use of the following result.
\begin{proposition}
\label{prop:compactemb}
Let $\Phi \in \cB_0$ satisfy Assumption \ref{assPGc} and define
$K_M=\{u \in H^\Phi(\R^d): \Norm{u}{\Phi} \le M\}$. There exists a constant $C$ such that
\begin{equation}\label{eq:equicont}
	\int_{\R^d}|u(x+h)-u(x)|^2dx \le C|h|^{2s}, \quad u \in K_M,
\end{equation}
for every $h \in \R^d$ with $0 \le |h| \le 1$.  Furthermore, if $\sequ u \subset H^{\Phi}(\R^d)$ with
$\sup_{n \in \N}\Norm{u_n}{\Phi}<\infty$, then there exists a function $u \in L^2(\R^d)$ and a non relabelled
subsequence $\sequ u$ such that the following hold:
\begin{enumerate}
\item
$u_n \rightharpoondown u$ in the weak topology of $L^2(\R^d)$;
\item
$u_n\mathbf{1}_K \to u \mathbf{1}_K$ in the strong topology of $L^2(\R^d)$, for every compact set $K \subset \R^d$;
\item
$u_n(x) \to u(x)$ for almost every $x \in \R^d$.
\end{enumerate}
\end{proposition}
\begin{proof}
(1) By the definition of $\Norm{\cdot}{\Phi}$, for every $u \in K_M$ we have $\Norm{u}{2} \le M$. Furthermore, by
Proposition \ref{prop:embed} there exists a constant $C_1$ such that for every $u \in K_M$ the bound $\vertiii{u}_s
\le C_1$ applies. Hence, by \cite[Lem. A.1]{BLP14} we have that
\begin{equation*}
\sup_{0 \le |h| \le 1}\int_{\R^d}\frac{|u(x+h)-u(x)|^2}{|h|^{2s}}dx \le C,
\end{equation*}
with a constant $C$ independent of $u \in K_M$. In particular, for every $h \in \R^d$ with $0 \le |h| \le 1$,
it follows that
\begin{equation*}
\int_{\R^d}|u(x+h)-u(x)|^2dx \le C|h|^{2s}, \quad u \in K_M.
\end{equation*}

Next consider (2) and let $\sequ u\subset H^{\Phi}(\R^d)$ and $M=\sup_{n \in \N}\Norm{u_n}{\Phi}$. Then $\Norm{u_n}{2}
\le M$ and there exists $u \in L^2(\R^d)$ such that a non-relabelled subsequence $\sequ u$ is such that $u_n
\rightharpoonup u$ in the weak topology of $L^2(\R^d)$. Furthermore, note that $\sequ u\subset K_M$, hence
\eqref{eq:equicont} holds. For any $\ell \in \N$, consider the ball $B_\ell$ and let $\eta_\ell \in C_{\rm c}^\infty(\R^d)$
be such that $\eta_\ell(x)=1$ for all $x \in B_\ell$, $\eta_\ell(x)=0$ for all $x \not \in \bar{B}_{\ell+1}$ and
$0 \le \eta_\ell(x) \le 1$ for all $x$. First, define $\widetilde{u}_n=\eta_1u_n$. Since they are all supported in
$\bar{B}_{2}$, the sequence is clearly equitight. Moreover, for every $h \in \R^d$ with $0 \le |h| \le 1$ we have
\begin{align}
\label{eq:equicont2}
\begin{split}
&\int_{\R^d}|\eta_1(x+h)u_n(x+h)-\eta_1(x)u_n(x)|^2dx\\
&\le \int_{\R^d}|\eta_1(x+h)-\eta_1(x)|^2|u_n(x+h)|^2dx + \int_{\R^d}|\eta_1(x)|^2|u_n(x+h)-u_n(x)|^2dx\\
& \le \left(\int_{\R^d}|u_n(x+h)|^2dx\right)\Norm{\nabla \eta_1}{\infty}^2|h| +\Norm{\eta_1}{\infty}^2
\int_{\R^d}|u_n(x+h)-u_n(x)|^2dx\\
&\le C(|h|+|h|^{2s}),			
\end{split}
\end{align}
where the constant $C$ is independent of $n$. Hence, by the Fr\'echet-Kolmogorov compactness theorem \cite[Cor. 4.27]{B11}
we obtain that there exists a subsequence $(\eta_1 u_n^{(1)})_{n\in\N}$ convergent in the strong topology of
$L^2(\R^d)$. In particular, this implies that $u_n^{(1)}$ converges in the strong topology of $L^2(B_1)$ and, since we
already know that $u_n \rightharpoonup u$, we also have $u_n^{(1)} \to u$ in $L^2(B_1)$, i.e., $\mathbf{1}_{B_1}u_n^{(1)}
\to \mathbf{1}_{B_1}u$ in $L^2(\R^d)$. Now assume that we already defined $u_n^{(\ell)}$ for $\ell \in \N$, and consider
$\widetilde{u}_n^{(\ell)}=\eta_{\ell+1}u_n^{(\ell)}$. Clearly, \eqref{eq:equicont2} still holds if we substitute
$\eta_{\ell+1}$ to $\eta_1$ and the sequence $\widetilde{u}_n^{(\ell)}$ is still equitight as they are all supported in
$B_{\ell+2}$. Thus there exists a convergent subsequence $(\eta_{\ell+1}u_n^{(\ell+1)})_{n\in\N}$ in the strong topology
of $L^2(\R^d)$. By a similar argument, we have $\mathbf{1}_{B_{\ell+1}}u_n^{(\ell+1)} \to \mathbf{1}_{B_{\ell+1}}u$ in
$L^2(\R^d)$. Clearly, for any compact set $K \subset \R^d$, the sequence $\sequ u$ is by construction such that
$\mathbf{1}_Ku_n^{(n)} \to \mathbf{1}_K u$ in $L^2(\R^d)$.

To obtain (3), let $\seq u$ be a subsequence as constructed in part (2) and reset the label. For any $\ell \in \N$ we
have $u_n \to u$ in $L^2(B_1)$ and thus there exists a set $E_1$ and a subsequence $(u_n^{(1)})_{n\in\N}$ such that
$|B_1 \setminus E_1|=0$ and $u_n^{(1)}(x) \to u(x)$ for all $x \in E_1$. Suppose we already defined the subsequence
$u_n^{(\ell)}$. Then $u_n^{(\ell)} \to u$ in $L^2(B_{\ell+1})$ and thus there exists a set $E_{\ell+1}$ and a subsequence
$u_n^{(\ell+1)}$ such that $|B_{\ell+1} \setminus E_{\ell+1}|=0$ and $u_n^{(\ell+1)}(x) \to u(x)$ for all $x \in E_{\ell+1}$.
The subsequence $(u_n^{(n)})_{n\in\N}$ satisfies all the properties in the statement.
\end{proof}

Another property we will use is the following.
\begin{proposition}
\label{thm:density}
The space $C_{\rm c}^\infty(\R^d)$ is dense in $H^\Phi(\R^d)$.
\end{proposition}
\noindent
This result follows from \cite[Th. 3.3]{HIL12} but we provide a more constructive proof directly using $j_\Phi$ in
Appendix 5.1 below. For later use we also note the following property.
\begin{lemma}
Let $\Phi \in \cB_0$. If $u,v \in H^{\Phi}(\R^d) \cap L^\infty(\R^d)$, then $uv \in H^{\Phi}(\R^d)$. Moreover,
the same holds if $v \in C^{\infty}_{\rm c}(\R^d)$ and $u \in H^{\Phi}(\R^d)\cap L^\infty_{\rm loc}(\R^d)$.
\end{lemma}
\begin{proof}
To show the first statement, observe that
	\begin{align*}
		[uv]_{\Phi}^2&=\int_{\R^d}\int_{\R^d}|u(x)v(x)-u(y)v(y)|^2j_{\Phi}(|x-y|)dxdy\\
		&\le 2\int_{\R^d}\int_{\R^d}|u(x)v(x)-u(x)v(y)|^2j_{\Phi}(|x-y|)dxdy\\
		&\qquad + 2\int_{\R^d}\int_{\R^d}|u(x)v(x)-u(y)v(x)|^2j_{\Phi}(|x-y|)dxdy
		\le 2(\Norm{u}{\infty}^2[v]_{\Phi}^2+\Norm{v}{\infty}^2[u]_{\Phi}^2).
	\end{align*}
To prove the second, let $K$ be the support of $v$ and split up the integral like
	\begin{align*}
	[uv]_{\Phi}^2&=\int_{\R^d}\int_{\R^d}|u(x)v(x)-u(y)v(y)|^2j_{\Phi}(|x-y|)dxdy\\
	&=\int_{K}\int_{K}|u(x)v(x)-u(y)v(y)|^2j_{\Phi}(|x-y|)dxdy\\
	&\qquad +2\int_{\R^d \setminus K}\int_{K}|u(x)v(x)|^2j_{\Phi}(|x-y|)dxdy
	=:I_1+I_2.
\end{align*}
Arguing as before, we get
\begin{equation*}
	I_1 \le (\Norm{u}{L^\infty(K)}^2[v]_{\Phi}^2+\Norm{v}{L^\infty(K)}^2[u]_{\Phi}^2).
\end{equation*}
To handle $I_2$, fix $\varepsilon>0$ and let $K_\varepsilon=\{x \in \R^d: \ {\rm dist}(x,K)\le \varepsilon\}$. Then
\begin{equation*}
I_2=2\int_{\R^d \setminus K_\varepsilon}\int_{K}|u(x)v(x)|^2j_{\Phi}(|x-y|)dxdy+2\int_{K_\varepsilon\setminus K}
\int_{K}|u(x)v(x)|^2j_{\Phi}(|x-y|)dxdy=I_3+I_4.
\end{equation*}
Furthermore,
\begin{equation*}
	I_3 \le 2\Norm{v}{\infty}^2\int_{K}|u(x)|^2\int_{\R^d \setminus K_\varepsilon}j_{\Phi}(|x-y|)dydx.
\end{equation*}
Clearly, if $x \in K$, we have $\R^d \setminus K_\varepsilon \subset \R^d \setminus B_\varepsilon(x)$ and then
\begin{equation*}
	I_3 \le 2\Norm{v}{\infty}^2\int_{K}|u(x)|^2\int_{\R^d \setminus B_\varepsilon(x)}j_{\Phi}(|x-y|)dydx
\le 2 \Norm{v}{\infty}^2\Norm{u}{2}^2\int_{\R^d \setminus B_\varepsilon}j_{\Phi}(|y|)dy<\infty.
\end{equation*}
Finally,  observe that if $y \in K_\varepsilon \setminus K$, then $u(x)v(y)=0$ and we have
\begin{align*}
	I_4&=2\int_{K_\varepsilon\setminus K}\int_{K}|u(x)|^2|v(x)-v(y)|^2j_{\Phi}(|x-y|)dxdy\\
	&\le 2\Norm{|\nabla v|}{\infty}^2\int_{K}|u(x)|^2\int_{K_\varepsilon\setminus K}|x-y|^2j_{\Phi}(|x-y|)dydx.
\end{align*}
Clearly, there exists $R>0$ such that $K_\varepsilon \setminus K \subset B_R(x)$ and then
\begin{align*}
	I_4&\le 2\Norm{|\nabla v|}{\infty}^2\int_{K}|u(x)|^2\int_{B_R(x)}|x-y|^2j_{\Phi}(|x-y|)dydx\\
	& \le 2\Norm{|\nabla v|}{\infty}^2\Norm{u}{2}^2\int_{B_R(x)}|y|^2j_{\Phi}(|y|)dy<\infty.
\end{align*}
\end{proof}

\subsection{Bernstein functions of the Laplacian}
Let $\Delta$ be the Laplacian and $\Phi \in \cB_0$. For any measurable function $u:\R^d \to \R$ we define
\begin{equation*}
\Phi(-\Delta)u(x)=\lim_{\varepsilon \to 0}\int_{\R^d\setminus B_\varepsilon(x)}(u(x)-u(y))j_{\Phi}(|x-y|)dy,
\end{equation*}
provided the involved quantities are finite. As $u \in H^1(\R^d)$, this expression is a direct consequence of
the Bochner-Phillips subordination formula (see, for instance, \cite[Ex. 11.6]{SSV}). An alternative definition
can be made via the following expression (for a proof see \cite[Prop. 2.6]{AL}).
\begin{proposition}
Let $\Phi \in \cB_0$, $x \in \R^d$ and a measurable function $u:\R^d \to \R$. Then
\begin{equation}
\label{eq:Phidelta}
\Phi(-\Delta)u(x)=-\frac{1}{2}\int_{\R^d}(u(x+h)-2u(x)+u(x-h))j_{\Phi}(|h|)dh,
\end{equation}
provided the integral at the right-hand side is convergent.
\end{proposition}
We want to establish a condition under which the integral at the right hand side of \eqref{eq:Phidelta} is
absolutely convergent. In \cite{AL} we considered a H\"older-Zygmund type condition. Here we aim to use a
more classical approach in order to prove that whenever $u \in H^{\Phi}(\R^d)$, we can define $\Phi(-\Delta)
u \in L^2(\R^d)$ by extension of a bounded linear operator. This property is well-known for the case of the
fractional Laplacian corresponding to $\Phi(z)=z^s$, see \cite[Sect. 3]{DPV12}. First we prove the following
technical lemma.
\begin{lemma}
\label{lem:existence1}
If $u \in C^2(\R^d) \cap L^\infty(\R^d)$, then the expression at the right-hand side of \eqref{eq:Phidelta}
is absolutely convergent for every $x \in \R^d$.
\end{lemma}
\begin{proof}
Fix $x \in \R^d$ and split up the integral in \eqref{eq:Phidelta} like
\begin{align*}
\int_{\R^d}|u(x+h)-2u(x)+u(x-h)|j_{\Phi}(|h|)dh
&= \left(\int_{B_1} + \int_{\R^d \setminus B_1}\right)|u(x+h)-2u(x)+u(x-h)|j_{\Phi}(|h|)dh\\
&=I_1+I_2.
\end{align*}
To estimate $I_1$, we note that due to $u \in C^2_{\rm b}(\R^d)$ there exists $C=C(x)>0$ such that
\begin{equation*}
|u(x+h)-2u(x)+u(x-h)|\le C|h|^2, \quad h \in B_1,
\end{equation*}
so that $I_1 \le C\int_{B_1}|h|^2j_\Phi(|h|)dh < \infty$, as a consequence of Lemma \ref{lem:propj}(3).
Also, we have $I_2 \le 4\Norm{u}{\infty}\int_{\R^d \setminus B_1}j_\Phi(|h|)dh< \infty$ by the same lemma.
\end{proof}
Next we further use an equivalent formulation of $\Phi(-\Delta)$, involving Schwartz space $\cS(\R^d)$. Taking
$u \in \cS(\R^d)$, we have
\begin{equation}
\label{moreequiv}
\cF[\Phi(-\Delta)u](\xi)=\Phi(|\xi|^2)\widehat{u}(\xi).
\end{equation}
This is a known result, but for the convenience of the reader we provide a proof in an appendix (Section 5.2
below), showing the precise relations with the set-up adopted in this paper. By a simple application of Plancherel's
theorem, Lemma \ref{lem:GagliardotoFourier} and Proposition \ref{thm:density}, we have then the following consequence.
\begin{corollary}
If $u \in \cS(\R^d)$, then
$||\sqrt{\Phi(-\Delta)}u||_{2}=[u]_{\Phi}$.
Furthermore, $\sqrt{\Phi(-\Delta)}$ can be extended uniquely to a self-adjoint bounded linear operator on
$H^\Phi(\R^d)$ by setting
\begin{equation*}
\sqrt{\Phi(-\Delta)}u=\cF^{-1}[\sqrt{\Phi(|\cdot|^2)}\widehat{u}], \quad u \in H^\Phi(\R^d).
\end{equation*}
\end{corollary}

Using the representation \eqref{moreequiv} specifically for the space $H^s(\R^d)$, $s>0$, we also show the
following for later use.
\begin{lemma}
\label{lem:product2}
Let $u \in H^s(\R^d) \cap L^1(\R^d)$ and $v \in C_{\rm c}^\infty(\R^d)$. Then $uv \in H^s(\R^d)\cap L^1(\R^d)$.
\end{lemma}
\begin{proof}
First observe that since $v \in C_{\rm c}^\infty(\R^d)$, then also $v \in L^1(\R^d)$. Furthermore, since $v \in
C_{\rm c}^\infty(\R^d)$, by the H\"older inequality we have $uv \in L^1(\R^d) \cap L^2(\R^d)$. By the convolution property
$\cF[uv]=\widehat{u} \ast \widehat{v}$ of Fourier transform we have
\begin{align*}
\int_{\R^d}|\xi|^{2s}|\cF[uv](\xi)|^2d\xi&=\int_{\R^d}|\xi|^{2s}\left|\int_{\R^d}\widehat{v}(y)\widehat{u}(\xi-y)dy\right|^2
d\xi\\
&\le \Norm{v}{\infty}^2\int_{{\rm supp}(v)}\int_{\R^d} |\xi|^{2s}|\widehat{u}(\xi-y)|^2d\xi dy \\
&=\Norm{v}{\infty}^2\int_{{\rm supp}(v)}\int_{\R^d} |\xi+y|^{2s}|\widehat{u}(\xi)|^2d\xi dy.
\end{align*}
Now let $R>0$ be such that ${\rm supp}(v)\subset B_R$. If $\xi \in \R^d \setminus B_{3R}$, then $\max_{y \in B_R}|\xi+y|^{2s}
=(|\xi|+R)^{2s} \le C|\xi|^{2s}$ for a constant $C>0$. Then we have
\begin{align*}
\int_{\R^d}|\xi|^{2s}|\cF[uv](\xi)|^2d\xi& \le \Norm{v}{\infty}^2\int_{B_R}\int_{B_{3R}} |\xi+y|^{2s}|
\widehat{u}(\xi)|^2d\xi dy \\
&\qquad
+C\Norm{v}{\infty}^2\int_{B_R}\int_{\R^d \setminus B_{3R}} |\xi|^{2s}|\widehat{u}(\xi)|^2d\xi dy\\
&\le \Norm{v}{\infty}^2\Norm{\widehat{u}(\xi)}{L^\infty(B_{3R})}^2R^{2(d+s)}4^{2(d+s)}3^{d}\omega_d^2 \\
&\qquad
+ C\Norm{v}{\infty}^2R^d\omega_d\int_{\R^d}|\xi|^{2s}|\widehat{u}(\xi)|^2d\xi<\infty.
\end{align*}
\end{proof}

\subsection{Dirichlet forms, non-local Schr\"odinger operators and ground states}
Clearly, the above description is not sufficient to guarantee the pointwise existence of $\Phi(-\Delta)u$ when
$u \in H^{\Phi}(\R^d)$, thus to handle integro-differential equations involving these operators a weak formulation
is needed. The operator $\Phi(-\Delta)$ naturally induces a Dirichlet form acting on $H^{\Phi}(\R^d) \times H^{\Phi}(\R^d)$
as
\begin{equation*}
\cE_{\Phi}(u,v)=\frac{1}{2}\int_{\R^d}\int_{\R^d}(u(y)-u(x))(v(y)-v(x))j_{\Phi}(|x-y|)dxdy.
\end{equation*}
The fact that $\cE_{\Phi}$ is well-defined on $H^{\Phi}(\R^d)\times H^{\Phi}(\R^d)$ is an easy consequence of
Young's inequality $2|ab|\le a^2+b^2$, since
\begin{equation*}
\int_{\R^d}\int_{\R^d}|u(y)-u(x)||v(y)-v(x)|j_{\Phi}(|x-y|)dxdy\le \frac{1}{2}([u]_{\Phi}^2+[v]_{\Phi}^2).
\end{equation*}
The link between the Dirichlet form $\cE_\Phi$ and the operator $\Phi(-\Delta)$ can be seen by choosing two functions
$u, v \in \cS(\R^d)$ and observing that
\begin{align*}
	\int_{\R^d}\int_{|x-y|>\varepsilon}&(u(x)-u(y))v(x)j_{\Phi}(|x-y|)dydx\\
	&=\frac{1}{2}\int_{\R^d}\int_{|x-y|>\varepsilon}(u(x)-u(y))v(x)j_{\Phi}(|x-y|)dxdy\\
	&\qquad +\frac{1}{2}\int_{\R^d}\int_{|x-y|>\varepsilon}(u(x)-u(y))v(x)j_{\Phi}(|x-y|)dxdy\\
	&=\frac{1}{2}\int_{\R^d}\int_{|x-y|>\varepsilon}(u(x)-u(y))v(x)j_{\Phi}(|x-y|)dxdy\\
	&\qquad -\frac{1}{2}\int_{\R^d}\int_{|x-y|>\varepsilon}(u(x)-u(y))v(y)j_{\Phi}(|x-y|)dxdy\\
	&=\frac{1}{2}\int_{\R^d}\int_{|x-y|>\varepsilon}(u(x)-u(y))(v(x)-v(y))j_{\Phi}(|x-y|)dxdy.
\end{align*}
Taking the limit as $\varepsilon \to 0$, by an application of the dominated convergence theorem at the right-hand side
we get
\begin{equation*}
\langle \Phi(-\Delta)u,v\rangle=\cE_\Phi(u,v).
\end{equation*}
Due to this identification, we can give the following alternative formulation of $\cE_{\Phi}$.
\begin{proposition}
\label{eq:DirichletFourier}
For every $u,v \in H^{\Phi}(\R^d)$ we have
\begin{equation*}
		\cE_\Phi(u,v)=\int_{\R^d}\Phi(|\xi|^2)\widehat{u}(\xi)\overline{\widehat{v}}(\xi)d\xi.
\end{equation*}
\end{proposition}
\begin{proof}
First consider $u,v \in \cS(\R^d)$. Then by Plancherel's theorem,
\begin{equation*}
\cE_{\Phi}(u,v)=\langle \Phi(-\Delta)u,v\rangle=\int_{\R^d}\cF[\Phi(-\Delta)u]\overline{\widehat{v}}(\xi)d\xi=
\int_{\R^d}\Phi(|\xi|^2)\widehat{u}(\xi)\overline{\widehat{v}}(\xi)d\xi.
\end{equation*}
Next let $u,v \in H^{\Phi}(\R^d)$ and take two sequences $\sequ u,\sequ v \subset C^\infty_{\rm c}(\R^d)$ such that
$u_n \to u$ and $v_n \to v$ in $H^{\Phi}(\R^d)$, which exist by Proposition \ref{thm:density}. Clearly, $\cE_{\Phi}$ is
a symmetric non-negative defined bilinear form. However, $\cE_{\Phi}(\cdot,\cdot)+\langle \cdot, \cdot \rangle$ is a
scalar product on $H^{\Phi}$ with induced norm given by $\Norm{u}{\Phi}$. In particular, by the Schwarz inequality,
\begin{align*}
	|\cE_{\Phi}(u,v)-\cE_{\Phi}(u_n,v_n)|&=|\cE_{\Phi}(u,v-v_n)+\cE_{\Phi}(u-u_n,v_n)|\\
	&\le |\cE_{\Phi}(u,v-v_n)+\langle u,v-v_n\rangle|+|\cE_{\Phi}(u-u_n,v_n)+\langle u-u_n,v_n\rangle|\\
	&\quad +|\langle u,v-v_n \rangle|+|\langle u_n,v-v_n \rangle|\\
	&\le \Norm{u}{\Phi}\Norm{v-v_n}{\Phi}+\Norm{u-u_n}{\Phi}\Norm{v_n}{\Phi}+\Norm{u}{2}\Norm{v-v_n}{2}\\
	&\quad +\Norm{u}{2}\Norm{v-v_n}{2}\\
	&\le 2(\Norm{u}{\Phi}\Norm{v-v_n}{\Phi}+\Norm{u-u_n}{\Phi}\Norm{v_n}{\Phi}).
\end{align*}
Since $v_n \to v$ in $H^{\Phi}(\R^d)$, there exists a constant $C>0$ such that for $n$ large enough we have
$\Norm{v_n}{\Phi}\le C$. Hence for sufficiently large $n$
\begin{align*}
|\cE_{\Phi}(u,v)-\cE_{\Phi}(u_n,v_n)|&\le C(\Norm{v-v_n}{\Phi}+\Norm{u-u_n}{\Phi}) \to 0, \quad n \to \infty,
\end{align*}
which implies $\lim_{n \to \infty}\cE_{\Phi}(u_n,v_n)=\cE_{\Phi}(u,v)$. Now observe that
\begin{equation}
\label{eq:EPhi}
\cE_{\Phi}(u_n,v_n)=\int_{\R^d}\Phi(|\xi|^2)\widehat{u}_n(\xi)\overline{\widehat{v}_n}(\xi)d\xi=:\langle \sqrt{\Phi(|\cdot|^2)}\widehat{u}_n,\sqrt{\Phi(|\cdot|^2)}\widehat{v}_n\rangle_{\mathbb{C}},
\end{equation}
where $\langle \cdot, \cdot \rangle_{\mathbb{C}}$ denotes the scalar product on the space of square-integrable complex
valued functions $L^2(\R^d;\mathbb{C})$. Again by Schwarz inequality
\begin{align*}
	\big|\langle &\sqrt{\Phi(|\cdot|^2)}\widehat{u}_n,\sqrt{\Phi(|\cdot|^2)}\widehat{v}_n\rangle_{\mathbb{C}}-\langle \sqrt{\Phi(|\cdot|^2)}\widehat{u},\sqrt{\Phi(|\cdot|^2)}\widehat{v}\rangle_{\mathbb{C}}\big|\\
	&\le \big|\langle \sqrt{\Phi(|\cdot|^2)}\widehat{u},\sqrt{\Phi(|\cdot|^2)}(\widehat{v}-\widehat{v}_n)\rangle_{\mathbb{C}}|
+|\langle \sqrt{\Phi(|\cdot|^2)}(\widehat{u}-\widehat{u}_n),\sqrt{\Phi(|\cdot|^2)}\widehat{v}_n\rangle_{\mathbb{C}}\big|\\
	&\le \Norm{\sqrt{\Phi(|\cdot|^2)}\widehat{u}}{\Phi}\Norm{\sqrt{\Phi(|\cdot|^2)}(\widehat{v}-\widehat{v}_n)}{\Phi}
+\Norm{\sqrt{\Phi(|\cdot|^2)}\widehat{v}_n}{\Phi}\Norm{\sqrt{\Phi(|\cdot|^2)}(\widehat{u}-\widehat{u}_n)}{\Phi}\\
	&=[u]_{\Phi}[v-v_n]_{\Phi}+[v_n]_{\Phi}[u-u_n]_{\Phi}\\
	&\le C([v-v_n]_{\Phi}+[u-u_n]_{\Phi}) \to 0 \quad \mbox{as $n\to\infty$},
\end{align*}
which implies
\begin{equation*}
\lim_{n \to \infty}\langle \sqrt{\Phi(|\cdot|^2)}\widehat{u}_n,\sqrt{\Phi(|\cdot|^2)}\widehat{v}_n\rangle_{\mathbb{C}}
=\langle \sqrt{\Phi(|\cdot|^2)}\widehat{u},\sqrt{\Phi(|\cdot|^2)}\widehat{v}\rangle_{\mathbb{C}}.
\end{equation*}
Taking the limit on both sides of \eqref{eq:EPhi} we get the desired result.
\end{proof}

Consider the heat kernel $p^\Phi_t(x)$ of $\Phi(-\Delta)$, defined via inverse Fourier transform as
\begin{equation*}
p_t^\Phi(x)=\frac{1}{(2\pi)^d}\int_{\R^d}e^{-ix\cdot \xi}e^{-t \Phi(|\xi|^2)}d\xi, \quad x \in \R^d,
\end{equation*}
and the $1$-resolvent kernel
\begin{equation*}
G_1^\Phi(x)=\int_0^{\infty}e^{-t}p_t^\Phi(x)dt, \quad x \in \R^d.
\end{equation*}
For $\Phi \in \cB_0$, we know (see, for instance, \cite[Sect. 3.1]{KSV12}) that there exists a probability measure
$g_\Phi(t,\cdot)$ on $(0,\infty)$ such that
\begin{equation*}
p^\Phi_t(x)=\frac{1}{(4\pi)^{d/2}}\int_{0}^{\infty}s^{-\frac{d}{2}}e^{-\frac{|x|^2}{4s}}g_\Phi(t,ds).
\end{equation*}
In particular, it is clear that $p^\Phi_t$ is radially non-increasing for all $t>0$ and $\Norm{p^\Phi_t}{1}=1$.
We will use a suitable class of functions that we will use as potentials, see \cite[Sect. 4.1]{HIL12},
\cite[Def. 4.280]{LHB}.
\begin{definition}
We say that a non-negative function $V$ belongs to \emph{$\Phi$-Kato class} $\cK^{\Phi}(\R^d)$ whenever
\begin{equation*}
\lim_{\delta \downarrow 0}\sup_{x \in \R^d}\int_{|x-y|<\delta} G_1^\Phi(x-y)V(y)dy=0.
\end{equation*}
Also, we say that a function $V:\R^d \to \bar{\R}$ is \emph{$\Phi$-Kato decomposable} if $V^- \in \cK^{\Phi}(\R^d)$
and $\mathbf{1}_{K}V^+ \in \cK^{\Phi}(\R^d)$ for every compact set $K \subset \R^d$, where $V^+(x)=\max\{0,V(x)\}$
and $V^-(x)=\max\{0,-V(x)\}$, $x \in \R^d$. We denote the set of $\Phi$-Kato decomposable functions by
$\cK_{\rm dec}^\Phi(\R^d)$.
\end{definition}
\noindent
Clearly, if $V$ is such that $V^- \in L^\infty(\R^d)$ and $V^+\in L^\infty_{\rm loc}(\R^d)$, then $V \in
\cK_{\rm dec}^\Phi(\R^d)$.

For any potential $V \in \cK_{\rm dec}^\Phi(\R^d)$ we define the non-local Schr\"odinger operator $\cH_{\Phi,V}$ by
\begin{equation*}
\cH_{\Phi,V}=\Phi(-\Delta)+V,
\end{equation*}
whose core is given by $C_{\rm c}^\infty(\R^d)$. Such an operator has a self-adjoint realization, see \cite[Th. 4.8]{HIL12}
and \cite[Th. 4.285]{LHB}).

We introduce the symmetric bilinear form
\begin{equation*}
	\cA_{\Phi,V}(u,v)=\cE_{\Phi}(u,v)+\langle Vu,v\rangle,
\end{equation*}
with core $C^\infty_{\rm c}(\R^d) \times C^\infty_{\rm c}(\R)$. Clearly, it is well-defined on $\cD(\cA_{\Phi,V}) \times
\cD(\cA_{\Phi,V})$, where
\begin{equation*}
\cD(\cA_{\Phi,V})=\left\{u \in H^{\Phi}(\R^d): \ \left|\int_{\R^d}V(x)u^2(x)dx\right|< \infty\right\}.
\end{equation*}
Weak solutions of the Schr\"odinger equation $\cH_\Phi u=f$ are then defined as functions $u \in H^{\Phi}(\R^d)$ such that
\begin{equation*}
\cA_{\Phi,V}(u,v)=\langle f,v\rangle, \quad v \in C^\infty_{\rm c}(\R^d),
\end{equation*}
holds. Consider the respectively associated quadratic forms, extended to $L^2(\R^d)$, defined by
\begin{equation*}
\cE_{\Phi}[u]=
\begin{cases} \cE_{\Phi}(u,u) & \mbox{if $u \in H^\Phi(\R^d)$} \\
\infty & \mbox{if $u \not \in H^\Phi(\R^d)$}
\end{cases}
\quad \mbox{and} \quad
\cA_{\Phi,V}[u]=\begin{cases} \cA_{\Phi,V}(u,u) & \mbox{if $u \in \cD(\cA_{\Phi,V})$} \\
		\infty & \mbox{if $u \not \in \cD(\cA_{\Phi,V})$}.
	\end{cases}
\end{equation*}
Among potentials $V \in \cK_{\rm dec}^\Phi(\R^d)$, there are two large subclasses of special interest:
\begin{itemize}
\item
\emph{confining potentials}:  $V \in L^1_{\rm loc}(\R^d)$ such that $V^- \in L^\infty(\R^d)$ and $\lim_{|x| \to
\infty}V(x)=\infty$; in this case, by a simple application of Rellich's criterion \cite[Th. 4.71]{LHB}
it follows that ${\rm Spec}(\cH_{\Phi,V})={\rm Spec}_{\rm d}(\cH_{\Phi,V})$
\item
\emph{decaying potentials}: $V \in L^{\infty,0}(\R^d) := \{h\in L^\infty(\R^d): \lim_{|x| \to \infty}h(x)=0\}$;
in this case $\cD(\cA_{\Phi,V})=H^{\Phi}(\R^d)$, ${\rm Spec}_{\rm ess}(\cH_{\Phi,V})=[0,\infty)$ and
${\rm Spec}_{\rm d}(\cH_{\Phi,V}) \not = \emptyset$ as soon as there exists $u \in H^{\Phi}(\R^d)$ such
that $\cA_{\Phi,V}[u]<0$.
\end{itemize}
Whenever ${\rm Spec}_{\rm d}(\cH_{\Phi,V}) \not = \emptyset$, the lowest lying eigenvalue $\lambda_0=
\min{\rm Spec}_{\rm d}(\cH_{\Phi,V})$, called ground state eigenvalue (or ground state energy) plays a special
role. Since the semigroup $\{e^{-t\cH_{\Phi,V}}: t \geq 0\}$ is positivity improving,
by the Perron-Frobenius theorem \cite[Th. 4.123]{LHB} the ground state eigenvalue $\lambda_0$ is simple and the
uniquely corresponding eigenfunction $\varphi_0$ has a strictly positive version, which we will use throughout.
Also, we choose $\Norm{\varphi_0}{2}=1$ by usual convention. Furthermore, as a direct consequence of the fact that
$e^{-t\cH_{\Phi,V}}: L^p(\R^d) \to L^q(\R^d)$, $t > 0$, is continuous for every $1 \le p,q \le \infty$, we know that
$\varphi_0 \in L^\infty(\R^d)$. Also, since $p_t^\Phi \in L^1(\R^d)$, it follows that $\varphi_0 \in C(\R^d)$. For
details we refer to \cite[Prop. 4.291]{LHB}.

Dirichlet forms and related operators can also be defined for cases when $\R^d$ is replaced by a sufficiently regular
bounded open set $\Omega \subset \R^d$. The non-local Dirichlet Laplacian $\Phi(-\Delta)_\Omega$ is then acting on
functions $u \in H^{\Phi}(\R^d)$ such that $u(x)=0$ for every $x \in \R^d \setminus \Omega$. Formally, the non-local
Dirichlet Laplacian acts on the closure $H_0^\Phi(\Omega)$ of $C_{\rm c}^\infty(\Omega)$ into $H^\Phi(\R^d)$. The
corresponding Dirichlet form $\cE_{\Phi,\Omega}$ can be extended to the whole space $L^2(\R)$ by writing
\begin{equation*}
\cE_{\Phi,\Omega}[u]=
\begin{cases}
\cE_{\Phi,\Omega}(u,u) & u \in H^\Phi_0(\Omega) \\
\infty & u \not \in H^\Phi_0(\Omega).
\end{cases}
\end{equation*}
We can define the eigenvalues of the non-local Dirichlet Laplacian as the values $\lambda$ such that
\begin{equation*}
\begin{cases}
\Phi(-\Delta)u(x)=\lambda u(x) & x \in \Omega\\
u=0 & x \not \in \Omega \quad \mbox{or equivalently} \quad \Phi(-\Delta)_\Omega u = \lambda u, \; u \in H^\Phi_0(\Omega).
\end{cases}
\end{equation*}

For our purposes below, we will need $\Omega=B_r$, $r>0$, and the corresponding principal Dirichlet eigenvalue $\lambda_r>0$.
Furthermore, denote by $f_r \in L^2(\R^d)$ the corresponding eigenfunction, such that $f_r \ge 0$ and $\Norm{f_r}{2}=1$, and
by $\cE_{\Phi,B_r}(f_r,f_r)$ the associated Dirichlet form. It is known \cite[Cor. 2.3]{AL89} that $f_r$ is a radial,
non-increasing function, and clearly ${\rm supp}(f_r)=B_r$.

We can use the Dirichlet eigenfunction to find a sufficient condition for the existence of a ground state of $\cH_{\Phi,V}$
as an operator on $L^2(\R^d)$.
\begin{proposition}
\label{eq:prop}
Let $V \in L^{\infty,0}(\R^d)$ and suppose that there exists $r>0$ such that $V(x) \le v < 0$ in $B_r$ and $\lambda_r-v<0$.
Then $\cH_{\Phi,V}$ has a ground state.
\end{proposition}
\begin{proof}
Note that $\cE_\Phi(f_r,f_r)=\cE_{\Phi,B_r}(f_r,f_r)=\lambda_r$, as $f_r \in H^{\Phi}(\R^d)$. Hence we have
\begin{equation*}
\cA_{\Phi,V}(f_r,f_r)=\lambda_r+\int_{B_r}V(x)f_r^2(x)dx \le \lambda_r-v<0.
\end{equation*}
\end{proof}
In the remainder of this paper we will work under a second standing assumption as follows.
\begin{assumption}
\label{ass:existence}
There exists a unique, strictly positive ground state of $\cH_{\Phi,V}$.
\end{assumption}

We also note an alternative (variational) characterization of the ground state eigenvalue. If $V \in \cK_{\rm dec}^\Phi(\R^d)$,
the operator $\cH_{\Phi,V}$ is self-adjoint on its domain, hence we can use Courant's min-max theorem to write
\begin{equation}
\label{eq:minmax}
\lambda_0=\min_{\substack{u \in H^{\Phi}(\R^d) \\ \Norm{u}{2}=1}}\cA_{\Phi,V}(u,u).
\end{equation}

\subsection{Modes of convergence}
Since the following notions of convergence will play a key role in this paper, we briefly recall some definitions.
\begin{definition}
\label{def:Gammaconv}
Let $X$ be a metric space, $\sequ F$ a sequence of functionals $F_n: X \to \bar{\R}$, $n \in \mathbb N$, and
$F:X \to \bar{\R}$ another functional. The sequence $\sequ F$ is said to be \emph{$\Gamma$-convergent} to $F$,
denoted $F=\Gamma-\lim_{n \to \infty}F_n$, whenever
\begin{itemize}
\item[(1)]
for every sequence $\sequ x \subset X$ it follows that $F(x) \le \liminf_{n \to \infty}F_n(x_n)$;
\item[(2)]
for every $x \in X$ there exists a sequence $\sequ x \subset X$ convergent to $x$ such that $F(x) \ge
\limsup_{n \to \infty}F_n(x_n)$.
\end{itemize}
If $X=L^2(\R^d)$, we denote $F={\rm s}\Gamma-\lim_{n \to \infty}F_n$ to say that $\sequ F$ is $\Gamma$-convergent to
$F$ in the strong topology of $L^2(\R^d)$, while we denote $F={\rm w}\Gamma-\lim_{n \to \infty}F_n$ if $\sequ  F$ is
$\Gamma$-convergent to $F$ in the weak topology of $L^2(\R^d)$.
\end{definition}

We also recall the following concept.
\begin{definition}
Let $X$ be a metric space. A functional $F:X \to \bar{\R}$ is called \emph{coercive} if for every $M>0$ there exists a
compact set $K_M$ such that if $x \not \in K_M$, then $F(x)>M$ (or equivalently, if there exists a compact set $K_M$
such that if $F(x) \le M$, then $x \in K_M$). A sequence of functionals $\sequ F$, $F_n: X \to \bar{\R}$, is called
\emph{equicoercive} if for every $M>0$ there exists a compact set $K_M$ (independent of $n$) such that, for all
$n \in \N$, the relation $F_n(x) \le M$ implies $x \in K_M$.
\end{definition}

We will also make use of the following form of convergence of self-adjoint operators (see, for instance, \cite[Sect. 6.6]{T09}).
Let $X$ be a Hilbert space and $A_n:\cD(A_n) \subset X \to X$ be a sequence of closed self-adjoint operators with dense core
$\cD(A_n)$. Assume $A: \cD(A)\subset X \to X$ is a closed self-adjoint operator with dense core $\cD(A) \subseteq \cD(A_n)$ for
every $n \in \N$. Let $\rho(A_n), \rho(A)$ be the resolvents respectively of $A_n$ and $A$, and consider $R_{A_n}:\lambda \in
\rho(A_n) \to (\lambda-A)^{-1} \in \cL(X)$, where $\cL(X)$ denotes the space of bounded linear operators of $X$ on itself.
\begin{definition}
Let $\sequ A$ be a sequence of self-adjoint operators.
\begin{enumerate}
\item
Convergence $A_n \to A$ holds in \emph{strong resolvent sense} if $\rho(A)\cap \left(\cap_{n \in \N} \rho(A_n)\right) \not =
\emptyset$ and there exists $\lambda \in \rho(A)\cap \left(\cap_{n \in \N} \rho(A_n)\right)$ such that $R_{A_n}(\lambda)u \to R_{A}
(\lambda)u$ in the strong topology of $X$, for every $u \in \cD(A)\subseteq \cD(A_n)$. We denote it as $A={\rm SR}-\lim_{n \to
\infty}A_n$.

\item
Convergence $A_n \to A$ holds in \emph{norm resolvent sense} if $\rho(A)\cap \left(\cap_{n \in \N} \rho(A_n)\right) \not = \emptyset$
and there exists $\lambda \in \rho(A)\cap \left(\cap_{n \in \N} \rho(A_n)\right)$ such that $R_{A_n}(\lambda) \to R_{A}(\lambda)$
in the strong topology of $\cL(X)$. We denote it as $A={\rm NR}-\lim_{n \to \infty}A_n$.
\end{enumerate}
\end{definition}
Recall that
if $A={\rm SR}-\lim_{n \to \infty}A_n$, then $R_{A_n}(\lambda)u \to R_A(\lambda)u$ in the strong topology of $X$, for all
$u \in \cD(A_n)\subseteq \cD(A)$ and every $\lambda \in \rho(A)\cap \left(\cap_{n \in \N} \rho(A_n)\right)$. Also, recall
that if $A={\rm NR}-\lim_{n \to \infty}A_n$, then $R_{A_n}(\lambda) \to R_A(\lambda)$ in the strong topology of $\cL(X)$ for
every $\lambda \in \rho(A)\cap \left(\cap_{n \in \N} \rho(A_n)\right)$. For a proof see \cite[Cor. 6.32]{T09}.

The relationship of interest in our context between these modes of convergence is summarized by the following; for a proof we refer
to \cite[Th. 13.6]{DM12}.
\begin{proposition}
\label{prop:GammatoSR}
Let $X$ be a Hilbert space with norm $\Norm{\cdot}{X}$ induced by its scalar product.  Also, let $\sequ {\mathcal A}$ be a sequence
of positive quadratic forms $\mathcal A_n: L^2(\R^d) \to \bar{\R}$, $n \in \mathbb N$, and $\mathcal A: L^2(\R^d) \to \bar{\R}$
another positive quadratic form. Furthermore let $\sequ A$ be the sequence of positive self-adjoint linear operators on $L^2(\R^d)$
correspondingly defined by the quadratic forms $\sequ {\mathcal A}$, and $A$ be the positive self-adjoint linear operator on
$L^2(\R^d)$ defined by the quadratic form $\mathcal A$. Then the following two properties are equivalent:
\begin{enumerate}
\item[(1)]
There exists a constant $C_{\cA}>0$ such that
\begin{equation*}
{\rm w}\Gamma-\lim_{n \to \infty}(\cA_n+C_{\cA}\Norm{\cdot}{X}^2)=
{\rm s}\Gamma-\lim_{n \to \infty}(\cA_n+C_{\cA}\Norm{\cdot}{X}^2)=(\cA+C_{\cA}\Norm{\cdot}{X}^2).
\end{equation*}
\item[(2)]
$\sequ A$ converges to $A$ in strong resolvent sense.
\end{enumerate}
 \end{proposition}

\section{Continuity of ground state eigenvalues}
In the following we prove the stability of the ground state eigenvalue with respect to suitable variations of the
potential. We will make use of the connection between $\Gamma$-convergence and strong resolvent convergence as given
in Proposition \ref{prop:GammatoSR}. We have the following result.
\begin{theorem}
\label{thm:stability}
Let $\Phi \in \cB_0$ satisfy Assumption \ref{assPGc} for some $s \in (0,1)$. Also, let $\seq V\subset \cK_{\rm dec}^{\Phi}
(\R^d)$ and $V \in \cK_{\rm dec}^{\Phi}(\R^d)$ be such that
\begin{enumerate}
\item
$\max\{\sup_{k \in \mathbb N}\Norm{V^-_k}{\infty},\Norm{V^-}{\infty}\}=:C_V < \infty$;
\item
$V_k(x) \to V(x)$ as $k\to\infty$ holds for almost every $x \in \R^d$;
\item
for every $u \in \cD(\cA_{\Phi,V})$ we have $\limsup_{k \to \infty}\int_{\R^d}V_k(x)u^2(x)dx  \le \int_{\R^d}V(x)u^2(x)dx$.
\end{enumerate}
Then $\cH_{\Phi,V}={\rm SR}-\lim_{k \to \infty}\cH_{\Phi,V_k}$.
\end{theorem}
\begin{proof}
For easing the notation, we write $\cA_{\Phi,k}:=\cA_{\Phi,V_k}$, $\cA_{\Phi}:=\cA_{\Phi,V}$, $\cH_{\Phi,k}:=\cH_{\Phi,V_k}$
and $\cH_{\Phi}:=\cH_{\Phi,V}$.
We show that for every $\varepsilon >0$
\begin{equation*}
{\rm w}\Gamma-\lim_{n \to \infty}(\cA_{\Phi,k}+(C_{V}+\varepsilon)\Norm{\cdot}{2}^2)={\rm s}\Gamma-\lim_{n \to \infty}(\cA_{\Phi,k} +(C_{V}+\varepsilon)\Norm{\cdot}{2}^2)=(\cA_{\Phi}+(C_{V}+\varepsilon)\Norm{\cdot}{2}^2).
\end{equation*}
First we check part (2) of Definition \ref{def:Gammaconv}. Consider $u \in L^2(\R^d)$. If $u \not \in \cD(\cA_{\Phi})$,
then $\cA_{\Phi}[u]=\infty$ and the statement is trivial. Thus consider $u \in \cD(\cA_\Phi) \subseteq H^\Phi(\R^d)
\subseteq H^s(\R^d)$. Then part (2) follows by observing that
\begin{align*}
\limsup_{k \to \infty}\cA_{\Phi,k}[u]+(C_V+\varepsilon)\Norm{u}{2}^2
&= [u]^2_{\Phi} +\limsup_{k \to \infty}\int_{\R^d}V_k(x)u^2(x)dx+(C_V+\varepsilon)\Norm{u}{2}^2\\
&\le [u]^2_{\Phi} +\int_{\R^d}V(x)u^2(x)dx+(C_V+\varepsilon)\Norm{u}{2}^2=\cA_{\Phi}[u].
\end{align*}
To see part (1) of Definition \ref{def:Gammaconv}, first consider $u_k \to u$ in the strong topology of $L^2(\R^d)$.
The statement is clear as soon as $\liminf_{k \to \infty}(\cA_{\Phi,k}[u_k]+(C_V+\varepsilon)\Norm{u_k}{2}^2)=\infty$.
With no loss of generality, we can consider a (non-relabelled) subsequence $u_k$ such that (1) $u_k(x) \to u(x)$ a.e.,
(2) $\liminf_{k \to \infty}\cA_{\Phi,k}[u_k]=\lim_{k \to \infty}\cA_{\Phi,k}[u_k]$, and (3) $\liminf_{k \to \infty}
\cE_{\Phi}[u_k]=\lim_{k \to \infty}\cE_{\Phi}[u_k]$. Then by definition of $C_V$, we get $(V_k(x)+C_V+\varepsilon)u_k^2(x)
\ge 0$ for a.e. $x \in \R^d$. Fatou's Lemma then gives
\begin{multline*}
\liminf_{k \to \infty}\cA_{\Phi,k}[u_k]+(C_V+\varepsilon)\Norm{u_k}{2}^2=\lim_{k \to \infty}\cE_{\Phi}[u_k]
+\liminf_{k \to \infty}\int_{\R^d}(V_k(x)+C_V+\varepsilon)u_k^2(x)dx \\
\ge \cE_{\Phi}[u]+\int_{\R^d}(V(x)+C_V+\varepsilon)u^2(x)dx=\cA_{\Phi}[u]+(C_V+\varepsilon)\Norm{u}{2}^2.
\end{multline*}
Thus ${\rm s}\Gamma-\lim_{k \to \infty}(\cA_{\Phi,k}+(C_V+\varepsilon)\Norm{\cdot}{2}^2)=(\cA_{\Phi}+
(C_V+\varepsilon)\Norm{\cdot}{2}^2)$.
Next consider $u_k \rightharpoonup u$ in the weak topology of $L^2(\R^d)$. Again, the statement is straightforward
whenever
the left-hand side above is infinite. Assume then that for all $k \in \N$ we have $\cA_{\Phi,k}[u_k] +
(C_V+\varepsilon)\Norm{u_k}{2}^2 \le M$ and, without loss,
$$
\lim_{k \to \infty}\cA_{\Phi,k}[u_k]+(C_V+\varepsilon)\Norm{u_k}{2}^2
=\cA_{\Phi}[u]+(C_V+\varepsilon)\Norm{u}{2}^2
$$
and $\lim_{k \to \infty}[u_k]_{\Phi}^2=[u]_{\Phi}^2$.
This, along with the fact that $V_k(x)+C_V \ge 0$ a.e., implies that
\begin{equation*}
[u_k]^2_{\Phi}+\varepsilon \Norm{u_k}{2}^2 \le \cA_{\Phi,k}[u_k]+(C_V+\varepsilon)\Norm{u_k}{2}^2 \le M
\end{equation*}
and then $\Norm{u_k}{\Phi} \le C$ where $C$ is independent of $k \in \N$. Proposition \ref{prop:compactemb} says that
there exists a (non-relabelled) subsequence $u_k$ converging almost everywhere to $u$. Thus by Fatou's lemma we get with
this subsequence (since $V_k +C_V+\varepsilon>0$ for every $k \in \N$) that
\begin{align*}
\liminf_{k \to \infty}(\cA_{\Phi,k}[u_k]+(C_V+\varepsilon)\Norm{u_k}{2}^2)
&=\liminf_{k \to \infty}\left([u_k]^2_{\Phi}+\int_{\R^d}(V_k+C_V+\varepsilon)|u_k(x)|^2dx\right) \\
&=[u]^2_{\Phi}+\liminf_{k \to \infty}\int_{\R^d}(V_k+C_V+\varepsilon)|u_k(x)|^2dx \\
& \ge \cA_{\Phi}[u]+(C_V+\varepsilon)\Norm{u}{2}^2.
\end{align*}
Hence ${\rm w}\Gamma-\lim_{k \to \infty}(\cA_{\Phi,k}+
(C_V+\varepsilon)\Norm{\cdot}{2}^2) =(\cA_{\Phi}+(C_V+\varepsilon)\Norm{\cdot}{2}^2)$.
\end{proof}

As a consequence of Theorem \ref{thm:stability} and \cite[Th. 6.38]{T09} (see also \cite{W80}) we have the following
result on the stability of the spectrum.
\begin{corollary}
\label{cor:stabilityspectrum}
\hspace{100cm}
\begin{enumerate}
\item
Under the assumptions of Theorem \ref{thm:stability} we have ${\rm Spec}(\cH_{\Phi,V})\subseteq \lim_{k \to \infty}{\rm Spec}
(\cH_{\Phi,V_k})$, in the sense that for any $\lambda \in {\rm Spec}(\cH_{\Phi,V})$ and for every $\varepsilon>0$ there exists
$l \in \N$ such that $(\lambda-\varepsilon,\lambda+\varepsilon) \cap {\rm Spec}(\cH_{\Phi,V_k})\not = \emptyset$ for every $k
\ge l$.
\item
In particular, if $\Spec_{\rm ess}(\cH_{\Phi,V_k})=\Spec_{\rm ess}(\cH_{\Phi,V})$ for all $k \in \N$, then for any $\lambda \in
{\Spec}_{\rm d}(\cH_{\Phi,V})$ there exists a sequence $\seq\lambda$ such that $\lambda_k \in {\Spec}_{\rm d}(\cH_{\Phi,V_k})$
and $\lambda_k \to \lambda$ as $k\to\infty$. Furthermore, for every $k \in \N$ there exists an eigenfunction $\varphi_k$ of
$\cH_{\Phi,V_k}$ at eigenvalue $\lambda_k$ such that $\varphi_k \to \varphi$ in $L^2(\R^d)$, where $\varphi$ is an eigenfunction
of $\cH_{\Phi,V}$ at eigenvalue $\lambda$.
\end{enumerate}
\end{corollary}
\begin{remark}
{\rm
\hspace{100cm}		
\begin{enumerate}
\item
We give two sufficient conditions implying Condition (3) in Theorem \ref{thm:stability}.
\begin{itemize}
\item[(3')]
Condition (3) holds if for any $u \in \cD(\cA_{\Phi})$ there exists $\widetilde{V} \geq 0$ such that $V_k \le \widetilde{V}$
a.e. for all $k \in \N$ and $\int_{\R^d} \widetilde{V}(x)u^2(x)dx<\infty$. In this case, it is a direct consequence of Fatou's
Lemma.
\item[(3'')]
Condition (3) holds if $V_k-V \to 0$ in $L^p(\R^d)$, where $p>1$ if $d=1$ and $s \ge 1/2$, otherwise $p \ge \frac{d}{2s}$.
Since $V^- \in L^\infty(\R^d)$ and $u \in \cD(\cA_{\Phi})$, in this case indeed we have
\begin{equation*}
[u]^2_{\Phi} \le \cA_{\Phi}[u]+\Norm{V^-}{\infty}\Norm{u}{2}^2<\infty,
\end{equation*}
which implies, due to Propositions \ref{thm:sobolevt}-\ref{prop:embed} and log-convexity of the $L^p$-norms, that
$u \in L^{2q}(\R^d)$ with $\frac{1}{p}+\frac{1}{q}=1$. Then
\begin{equation*}
\label{eq:limitVk}
\int_{\R^d}|V_k(x)-V(x)|u^2(x)dx \le \Norm{V_k-V}{p}\Norm{u}{2q}^2 \to 0.
\end{equation*}
\end{itemize}
\item		
We note that the assumption $V \in \cK_{\rm dec}^{\Phi}(\R^d)$ is not necessary as long as there is an independent condition
guaranteeing that $\cH_{\Phi}$ is self-adjoint. For instance, if we consider $\cH_{\Phi}$ as the non-local Dirichlet Laplacian
of a smooth bounded open set $\Omega \subset \R^d$, we have $\cD(\cE_{\Phi,\Omega})=H_0^{\Phi}(\R^d)$. Consider
$$
V(x)=
\begin{cases} 0 & x \in \Omega \\
\infty & x \not \in \Omega
\end{cases}
$$
and $\cD(\cA_{\Phi,V})=H_0^{\Phi}(\R^d)$. Thus the non-local Dirichlet Laplacian can also be seen as a Schr\"odinger operator
with a degenerate potential. Consider the sequence of the anharmonic oscillators $V_k(x)= |x|^{2k}$. It is straightforward
to see that $V_k \le \widetilde{V}:= V +1$
and $V_k \to V$ almost everywhere. Furthermore, $\int_{\R^d}\widetilde{V}(x)u^2(x)dx=\Norm{u}{2}$ for all $u \in H_0^\Phi(\Omega)$.
Hence we can use Theorem \ref{thm:stability} and Corollary \ref{cor:stabilityspectrum} to conclude that the eigenvalues of the
sequence of anharmonic oscillators converge to the eigenvalues of the non-local Dirichlet Laplacian (and the same applies for the
eigenfunctions).

\item
The proof of Theorem \ref{thm:stability} also shows that $\cA_{\Phi,k}+(C_V+\varepsilon)\Norm{\cdot}{2}^2$ are
equicoercive. In particular, along the same proof we can show that $\cA_{\Phi,k}$ are equicoercive on any bounded subset of
$L^2(\R^d)$ and, if $\inf_{k \in \N}\inf_{x \in \R^d}V_k(x)>0$, then $\cA_{\Phi,k}$ are equicoercive in $L^2(\R^d)$.
\item
Finally, observe that Corollary \ref{cor:stabilityspectrum} can be used to prove the existence of a ground state. Indeed, if
for instance $V_k,V \in L^{\infty,0}(\R^d)$ and $\cH_{\Phi,k}$  admit a sequence of ground state eigenvalues $\lambda_{0,k}$ with
$\sup_{k \in \N}\lambda_{0,k}<0$, then $\cH_{\Phi}$ admits a ground state $\lambda_0<0$ defined as $\lim_{k \to \infty}\lambda_{0,k}$.
\item
The same result holds also in the classical Schr\"odinger operators featuring the Laplacian, by substituting the space $H^{\Phi}(\R^d)$
by the Sobolev space $H^1(\R^d)$ and $[u]_{\Phi}^2$ by $\Norm{\nabla u}{L^2(\R^d)}^2$.

\item
To highlight the importance of assumption (3) in Theorem \ref{thm:stability} in the classical case, consider the following
well-known example. Let $d=1$, $V_k(x)=\omega x^2+\frac{1}{k}|x|^{-3}$ and $V(x)=\omega x^2$. Clearly, $V_k(x) \to V(x)$ as $k \to
\infty$, for all $x \in \R \setminus \{0\}$, and $V_k,V \ge 0$. However, with $\eta \in C_{\rm c}^\infty(\R)$ such that $\eta(x)=1$
if $x \in [-1,1]$, $\eta(x)=0$ if $x \not \in [-2,2]$, and $0 \le \eta(x) \le 1$ in general, $\eta \in \cD(\cA_{\Phi,V})$ and clearly
$\int_{\R}V(x)\eta^2(x)dx<\infty$, while $\int_{\R}V_k(x)\eta^2(x)dx=\infty$. In the classical case, indeed it is well known that
${\rm SR}-\lim_{k \to \infty}(-\Delta+V_k) \neq -\Delta + V$. This is known as the Klauder phenomenon \cite{DH74,S73}, whose non-local
counterpart we will discuss elsewhere.
\end{enumerate}
}
\end{remark}

The same conclusion of Theorem \ref{thm:stability} can be obtained without Condition (3) under a stronger assumption on the
convergence of $\seq V$.
\begin{theorem}
\label{thm:stability3}
Let $\Phi \in \cB_0$ satisfy Assumption \ref{assPGc} for some $s \in (0,1)$. Also, let $\seq V\subset \cK_{\rm dec}^\Phi(\R^d)$
and $V \in \cK_{\rm dec}^\Phi(\R^d)$ be such that for every compact set $K \subset \R^d$ we have $\lim_{k\to\infty}|V_k-V|
\mathbf{1}_K = 0$ in $L^2(\R^d)$. Then $\cH_{\Phi,V}={\rm SR}-\lim_{k \to \infty}\cH_{\Phi,V_k}$.
\end{theorem}
\begin{proof}
Let $u \in C_{\rm c}^\infty(\R^d)$ and consider
\begin{equation*}
\Norm{\cH_{\Phi,V}u-\cH_{\Phi,V_k}u}{2}=\int_{\R^d}(V(x)-V_k(x))^2u^2(x)dx.
\end{equation*}
Write $K={\rm supp}(u)$ and observe that
\begin{equation*}
\Norm{\cH_{\Phi,V}u-\cH_{\Phi,V_k}u}{2}=\int_{K}(V(x)-V_k(x))^2u^2(x)dx \le \Norm{u^2}{\infty}\Norm{(V-V_k)\mathbf{1}_{K}}{2}^2.
\end{equation*}
Taking the limit we have $\cH_{\Phi,V_k}u \to \cH_{\Phi,V}u$ in $L^2(\R^d)$. Since $C_{\rm c}^\infty(\R^d)$ is a core for
$\cH_{\Phi,V_k}$ and $\cH_{\Phi,V}$, the proof follows by \cite[Lem. 6.36]{T09}.
\end{proof}
\begin{remark}
{\rm
Clearly, Corollary \ref{cor:stabilityspectrum} continues to hold also under the assumptions of Theorem \ref{thm:stability3}.
}
\end{remark}

A further strong convergence condition on the potentials yields the following.
\begin{theorem}
\label{thm:stability4}
Let $\Phi \in \cB_0$ satisfy Assumption \ref{assPGc} for some $s \in (0,1)$. Also, let $\seq V\subset
\cK_{\rm dec}^{\Phi}(\R^d)$ and $V \in \cK_{\rm dec}^{\Phi}(\R^d)$ be such that the convergence $|V_k-V| \to 0$ holds
uniformly. Then $\cH_{\Phi,V}={\rm NR}-\lim_{k \to \infty}\cH_{\Phi,V_k}$, in particular, ${\Spec}(\cH_{\Phi,V})=
\lim_{k \to \infty}{\Spec}(\cH_{\Phi,V_k})$.
\end{theorem}
\begin{proof}
Let $u \in C_{\rm c}^\infty(\R^d)$ and note that
\begin{equation*}
\Norm{\cH_{\Phi,V}u-\cH_{\Phi,V_k}u}{2}=\int_{\R^d}(V(x)-V_k(x))^2u^2(x)dx \le \Norm{V-V_k}{\infty}^2\Norm{u}{2}^2.
\end{equation*}
Since $\Norm{V-V_k}{\infty}^2 \to 0$ by assumption, the proof is immediate by \cite[Lem. 6.34, Th. 6.38]{T09}.
\end{proof}
In this case, Corollary \ref{cor:stabilityspectrum} can be slightly improved by using the cited theorem in the
above proof.
\begin{corollary}
Under the assumptions of Theorem \ref{thm:stability4} we have $\lim_{k \to \infty}{\rm Spec}(\cH_{\Phi,V_k})=
{\rm Spec}(\cH_{\Phi,V})$.
\end{corollary}

\section{Regularity and monotonicity of ground states for relativistic Schr\"odinger operators with spherical potential wells}
\subsection{Approximant ground states}
As an application of the previous stability results we now discuss important consequences on the properties of
ground states of a key family of non-local Schr\"odinger operators. Let $\alpha \in (0,2)$ and $m \ge 0$, and define
the function
\begin{equation*}
\Phi_{m,\alpha}(z)=(z+m^{2/\alpha})^{\alpha/2}-m, \quad z > 0.
\end{equation*}
It is straightforward to show that $\Phi_{m,\alpha} \in \cB_0$. Furthermore, if $m=0$, we have
\begin{equation*}
j_{0,\alpha}(r):=j_{\Phi_{0,\alpha}}(r)=\frac{2^{\alpha}\Gamma\left(\frac{d+\alpha}{2}\right)}{\pi^{d/2}
\left|\Gamma\left(-\frac{\alpha}{2}\right)\right|}\frac{1}{r^{d+\alpha}},
\end{equation*}
while if $m>0$, then
\begin{equation*}
j_{m,\alpha}(r):=j_{\Phi_{m,\alpha}}(r)=\frac{2^{\frac{\alpha-d}{2}}m^{\frac{d+\alpha}{2\alpha}}\alpha}{\pi^{d/2}
\Gamma\left(1-\frac{\alpha}{2}\right)}r^{-\frac{d+\alpha}{2}}K_{\frac{d+\alpha}{2}}(m^{1/\alpha}r),
\end{equation*}
where
\begin{equation*}
K_{\xi}(z)=\frac{1}{2}\left(\frac{z}{2}\right)^\xi \int_0^{\infty}t^{-\xi-1}e^{-t-\frac{z^2}{4t}}dt, \quad z>0, \
\xi>-\frac{1}{2}
\end{equation*}
is the modified Bessel function of the third kind. In the case $m=0$, the operator $\Phi_{0,\alpha}(-\Delta)$ is the
fractional Laplacian and $H^{\Phi_{0,\alpha}}(\R^d)=H^{\alpha/2}(\R^d)$ since then $[u]_{\Phi_{0,\alpha}}=
C_{\alpha,d}[[u]]_{\alpha/2}$. In the case $m>0$, the operator $\Phi_{m,\alpha}(-\Delta)$ is called \textit{relativistic
fractional Laplacian}. This terminology is natural since for $\alpha=1$ the operator coincides with the square-root
Klein-Gordon operator (in conventional units)
giving the kinetic term of a (semi-)relativistic particle. For a unified notation we use $\Phi_{m,\alpha}$ and
$j_{m,\alpha}$ subsuming the case $m=0$, and make the appropriate distinctions between massive and massless cases when
necessary.

It is clear, either by direct computation or by using the
fact that $\Phi_{m,\alpha}(\lambda) \sim \lambda^{\alpha/2}$ as $\lambda \to \infty$ and \cite[Th. 3.4]{KSV12},
that $j_{m,\alpha}$ satisfies Assumption \ref{assPGc} with $s=\frac{\alpha}{2}$ for all $m\geq 0$. Moreover, since for
$m>0$ we have $j_{m,\alpha}(r)\sim Cr^{-d-\alpha}$ as $r \to 0+$, for a constant $C>0$ depending on $d,m,\alpha$, it is
not difficult to check that $H^{\Phi_{m,\alpha}}(\R^d)=H^{\alpha/2}(\R^d)$ up to equivalence of norms. We denote
$\cE_{m,\alpha}:=\cE_{\Phi_{m,\alpha}}$, which acts on $H^{\alpha/2}(\R^d) \times H^{\alpha/2}(\R^d)$.

Furthermore, for the remainder of the paper we choose the specific potential
$$
V(x)=-v\mathbf 1_{B_a}(x),
$$
where $v,a>0$ are constants. By Proposition \ref{eq:prop} we can choose $v$ large enough so that $V$ satisfies Assumption
\ref{ass:existence}. For every $\varepsilon>0$ let $\eta_\varepsilon$ be a radially decreasing cutoff function for $B_a$
with support contained in $B_{a+\varepsilon}$. To construct such a function, we may take
\begin{equation*}
	\varrho(x)=\begin{cases} C_{\varrho}e^{-\frac{1}{1-|x|^2}} & \mbox{if $|x| \le 1$} \\
		0 & \mbox{if $|x|>1$},
	\end{cases}
\end{equation*}
where $C_\varrho$ is chosen in such a way that $\int_{\R^d}\varrho(x)dx=1$. Then define $\varrho_{\varepsilon/2}(x)=
(2\varepsilon)^{-d}\varrho\left(\frac{2 x}{\varepsilon}\right)$ and consider $\eta_\varepsilon=\varrho_{\varepsilon/2}
\ast \mathbf 1_{B_{a+\frac{\varepsilon}{2}}}$. It is clear that $\eta_\varepsilon(x)=1$ for every $x \in B_a$ and
$\eta_\varepsilon(x)=0$ for all $x \in \R^d \setminus B_{a+\varepsilon}$. To see that it is radially symmetric, let
${\mathrm R} \in {\mathrm {SO}}(d)$ be any rotation on the space $\R^d$ and observe that, since ${\mathrm R}$ is an
isometry,
\begin{align*}
\eta_\varepsilon({\mathrm R}x)
&=
\int_{\R^d}\varrho_{\varepsilon/2}({\mathrm R}x-y)\mathbf 1_{B_{a+\frac{\varepsilon}{2}}}(y)dy
=
\int_{\R^d}\varrho_{\varepsilon/2}({\mathrm R}x-{\mathrm R}y)\mathbf 1_{B_{a+\frac{\varepsilon}{2}}}({\mathrm R}y)dy \\
&=
\int_{\R^d}\varrho_{\varepsilon/2}(x-y)\mathbf 1_{B_{a+\frac{\varepsilon}{2}}}(y)dy=\eta_\varepsilon(x).
\end{align*}
Then taking the profile function $\widetilde{\eta}_\varepsilon:r \in [0,\infty) \mapsto \eta_\varepsilon(r\mathbf{e}_1)
\in \R$, we have for every $r>0$,
\begin{equation*}
\widetilde{\eta}_\varepsilon'(r)=\frac{\partial \, \varrho_{\varepsilon/2}}{\partial \, x_1} \ast
\mathbf 1_{B_a+\frac{\varepsilon}{2}}(r\mathbf{e}_1) \le 0,
\end{equation*}
(where $\mathbf{e}_1$ is a unit vector in $\R^d$)
implying that $\eta_\varepsilon$ is radially decreasing. Hence we can define
$$
V_\varepsilon(x)=-v\eta_\varepsilon(x)
$$
so that $V_\varepsilon \to V$ as $\varepsilon \to 0$ in $L^p(\R^d)$ for any $1 \le p < \infty$, and $V_\varepsilon$
satisfies Assumption \ref{ass:existence} for every $\varepsilon>0$. Furthermore, we also have
$V_\varepsilon \in C^\infty_{\rm c}(\R^d)$. With these entries we then define
$$
\cH_{m,\alpha}:=\cH_{\Phi_{m,\alpha},V} \quad \mbox{and} \quad \cH^\varepsilon_{m,\alpha}:=\cH_{\Phi_{m,\alpha},V_\varepsilon}.
$$
To avoid multiple subscripts, in this section we denote the unique ground states and ground state eigenvalues of these
non-local Schr\"odinger operators simply by
$$
\varphi, \;  \varphi_{\varepsilon} \quad \mbox{and} \quad \lambda, \; \lambda_{\varepsilon}, \;\; \mbox{respectively}.
$$


\subsection{Regularity of the ground state}
First we show that $\varphi_{\varepsilon}$ is regular to a high degree. Since the properties discussed just before
introducing Assumption \ref{ass:existence} also hold for the semigroup $\{e^{-t\cH^\varepsilon_{m,\alpha}}: t \geq 0\}$, we have
$\varphi_{\varepsilon} \in L^\infty(\R^d) \cap C(\R^d)$. Also, $\varphi_{\varepsilon} \in L^1(\R^d)$ by \cite[Prop. 4.291]{LHB}.

\begin{proposition}
\label{prop:regular}
We have that $\varphi_{\varepsilon} \in C^\infty(\R^d)\cap L^\infty(\R^d)$.
\end{proposition}
\begin{proof}
By the above $\varphi_{\varepsilon} \in C(\R^d) \cap L^\infty(\R^d)$. Also, we know by definition that $\varphi_{\varepsilon}
\in H^{\Phi_{m,\alpha}}(\R^d)$ and
\begin{equation*}
\cE_{m,\alpha}(\varphi_{\varepsilon},v)=\lambda \langle \varphi_{\varepsilon},v\rangle
+\langle V\varphi_{\varepsilon},v\rangle, \quad v \in C^\infty_{\rm c}(\R^d).
\end{equation*}
By the continuity of the involved operators, the above equality holds also for $v \in H^{m,\alpha}(\R^d)$. Define
the space
$$
\bar{H}^{m,\alpha}(\R^d) = \big\{u \in L^2(\R^d;\mathbb{C}): [\Re(u)]_{\Phi}^2 +[\Im(u)]_{\Phi}^2<\infty\big\}.
$$
On $\bar{H}^{m,\alpha}(\R^d)$ we define the bilinear form
\begin{equation*}
\bar{\cE}_{m,\alpha}(u,v)
=\frac{1}{2}\int_{\R^d}\int_{\R^d}(u(y)-u(x))(\bar{v}(y)-\bar{v}(x))j_{m,\alpha}(|x-y|)dxdy.
\end{equation*}
Assuming $u \in H^{\Phi}(\R^d)$ and noting that the real and imaginary parts of $v \in \bar{H}^{m,\alpha}(\R^d)$
are such that $\Re(v),\Im(v) \in H^{m,\alpha}(\R^d)$, we have
\begin{align*}
\bar{\cE}_{m,\alpha}(u,v)&=\frac{1}{2}\int_{\R^d}\int_{\R^d}(u(y)-u(x))(\overline{v}(y)-\overline{v}(x))
j_{m,\alpha}(|x-y|)dxdy\\
&=\frac{1}{2}\int_{\R^d}\int_{\R^d}(u(y)-u(x))(\Re(v)(y)-i\Im(v)(y)-\Re(v)(x)+i\Im(v)(x))j_{m,\alpha}(|x-y|)dxdy\\
&=\frac{1}{2}\int_{\R^d}\int_{\R^d}((u(y)-u(x))(\Re(v)(y)-\Re(v)(x))j_{m,\alpha}(|x-y|)\\
&\qquad +i(u(y)-u(x))(\Im(v)(y)-\Im(v)(x))j_{m,\alpha}(|x-y|))dxdy\\
&=\cE_{m,\alpha}(u,\Re(v))-i\cE_{m,\alpha}(u,\Im(v)).
\end{align*}
Choosing $u=\varphi_{\varepsilon}$ gives
	\begin{align*}
		\bar{\cE}_{m,\alpha}(\varphi_{\varepsilon},v)&=\cE_{m,\alpha}(u,\Re(v))-i\cE_{m,\alpha}(u,\Im(v))\\
		&=\lambda \langle \varphi_{\varepsilon},\Re(v)\rangle +\langle V\varphi_{\varepsilon},\Re(v)\rangle-i\lambda
\langle \varphi_{\varepsilon},\Im(v)\rangle -i\langle V\varphi_{\varepsilon},\Im(v)\rangle\\
		&=\lambda \langle \varphi_{\varepsilon},v\rangle_{\mathbb{C}} +\langle V\varphi_{\varepsilon},v\rangle_{\mathbb{C}}.
	\end{align*}
for all $v \in \bar{H}^{m,\alpha}(\R^d)$, in particular, for all $v \in \cS(\R^d;\mathbb{C}) \subset \bar{H}^{m,\alpha}(\R^d)$.
By linearity of the Fourier transform, Proposition \ref{eq:DirichletFourier} holds also for $\bar{\cE}_{m,\alpha}$ and thus by
Plancherel's theorem we have
\begin{equation*}
\int_{\R^d}\Phi_{m,\alpha}(|\xi|^2)\widehat{\varphi}_{\varepsilon}(\xi)\overline{\widehat{v}}(\xi)d\xi
=\lambda\int_{\R^d}\widehat{\varphi_{\varepsilon}}(\xi)\overline{\widehat{v}}(\xi)d\xi
+\int_{\R^d} \cF[V\varphi_{\varepsilon}](\xi)\overline{\widehat{v}}(\xi)d\xi.
\end{equation*}
In particular, since $C_{\rm c}^{\infty}(\R^d) \subset \cS(\R^d;\mathbb{C})$ and $\cF^{-1}(\cS(\R^d;\mathbb{C}))=
\cS(\R^d; \mathbb{C})$, for every function $\psi \in C_{\rm c}^{\infty}(\R^d)$, there exists a function $v \in
\cS(\R^d;\mathbb{C})$ such that $\widehat{v}=\psi$ and then
\begin{equation*}
\int_{\R^d}(\Phi_{m,\alpha}(|\xi|^2)\widehat{\varphi}_{\varepsilon}(\xi)
-\lambda\widehat{\varphi_{\varepsilon}}(\xi)-\cF[V\varphi_{\varepsilon}](\xi))\psi(\xi)d\xi=0.
\end{equation*}
By the fundamental lemma of variational calculus (see, e.g., \cite [Th. 1.24]{D14}), we obtain
\begin{equation}
\label{eq:equalityFourier}
\Phi_{m,\alpha}(|\xi|^2)\widehat{\varphi}_{\varepsilon}(\xi)-\lambda\widehat{\varphi_{\varepsilon}}(\xi)-
\cF[V\varphi_{\varepsilon}](\xi)=0, \quad \mbox{a.e. $\xi \in \R^d$.}
\end{equation}
This relation implies that $\Phi_{m,\alpha}(|\xi|^2)\widehat{\varphi}_{\varepsilon} \in L^2(\R)$ and thus
$\Phi_{m,\alpha}(-\Delta)\varphi_{\varepsilon} \in L^2(\R)$. To prove that $\varphi_{\varepsilon} \in C^\infty(\R^d)$,
we only need to show that for any $n \in \N$ the function $|\xi|^{n\alpha}\widehat{\varphi}_{\varepsilon}(\xi)$
belongs to $L^2(\R)$ (see, for instance, \cite[Exercise 6, Sect. 2.3]{DM72}). To do this, observe that
$\Phi_{m,\alpha}(|\xi|^2)\sim |\xi|^{\alpha}$ as $|\xi| \to \infty$. Combining this asymptotic behaviour with the
fact that $\widehat{\varphi}_{\varepsilon}$ is uniformly continuous and $\Phi_{m,\alpha}(|\xi|^2)
\widehat{\varphi}_{\varepsilon} \in L^2(\R)$, we get that $|\xi|^{\alpha}\widehat{\varphi}_{\varepsilon} \in
L^2(\R)$.
	
Now assume that $|\xi|^{n\alpha}\widehat{\varphi}_{\varepsilon}(\xi)$ belongs to $L^2(\R)$ for some $n \in \N$.
By Lemma \ref{lem:product2}, we also know that $|\xi|^{n\alpha}\cF[V_{\varepsilon}\varphi_{\varepsilon}](\xi) \in
L^2(\R^d)$. Note that $V_{\varepsilon}\varphi_{\varepsilon} \in L^1(\R^d)$ by the H\"older inequality, hence
$\cF[V_{\varepsilon}\varphi_{\varepsilon}](\xi)$ is uniformly continuous. Furthermore, we have
\begin{equation*}
\widehat{\varphi}(\xi)=\frac{\cF[V_\varepsilon \varphi_{\varepsilon}](\xi)}{\Phi_{m,\alpha}(|\xi|^2)-\lambda},
\end{equation*}
where, recalling that $\lambda<0$, we have $\Phi_{m,\alpha}(|\xi|^2)-\lambda>0$. Multiplying both sides by
$|\xi|^{(n+1)\alpha}$, taking the square and integrating, we obtain
\begin{align*}
\int_{\R^d}|\xi|^{2(n+1)\alpha}|\widehat{\varphi}(\xi)|^2d\xi&
=\left(\int_{B_1}+\int_{\R^d\setminus B_1}\right)\frac{|\xi|^{2(n+1)\alpha}|\cF[V_\varepsilon \varphi_{\varepsilon}](\xi)|^2}
{(\Phi_{m,\alpha}(|\xi|^2)-\lambda)^2}d\xi\\
&\le \frac{\Norm{|\cF[V_\varepsilon \varphi_{\varepsilon}]|^2}{L^\infty(B_1)}\omega_d}{-\lambda}
+\frac{1}{C}\int_{\R^d\setminus B_1}|\xi|^{2n\alpha}|\cF[V_\varepsilon \varphi_{\varepsilon}](\xi)|^2d\xi<\infty,
\end{align*}
where $C=\inf_{|\xi|>1}\frac{\Phi(|\xi|^2)-\lambda}{|\xi|^\alpha}>0$.
\end{proof}

\begin{remark}
\label{rmk:continuity}
\rm{	
The argument we used to prove that $\varphi_{\varepsilon} \in C(\R^d)$ applies more generally. Indeed, we can prove
in the same way that if $\Phi \in \cB_0$ and $V \in L^{\infty,0}(\R^d)$ are such that Assumptions \ref{assPGc}(1)
and \ref{ass:existence} are satisfied, the function $p_t^{\Phi}(x)= (2\pi)^{-d}\int_{\R^d}e^{-ix \cdot \xi-t\Phi(|\xi|^2/2)}
d\xi$ is radially symmetric, and $\varphi$ is the ground state of $\cH_{\Phi,V}$, then $\varphi \in C(\R^d) \cap
L^\infty(\R^d)$. This can be further extended to $\Phi$-Kato-decomposable potentials.
}
\end{remark}
Now that we know that $\varphi_{\varepsilon} \in C^\infty(\R^d) \cap L^\infty(\R^d)$, Lemma \ref{lem:existence1}
guarantees that $\Phi_{m,\alpha}(-\Delta)\varphi_{\varepsilon}$ is well defined as in \eqref{eq:Phidelta}. Using
Theorem \ref{thm:stability} we see that $\varphi_{\varepsilon}$ are in a sense approximants of $\varphi$.
\begin{proposition}
\label{appx}
There exists a sequence $\seq \varepsilon$, $\varepsilon_k \downarrow 0$, such that $\varphi_{k}:=\varphi_{\varepsilon_k}
\to \varphi$ both in $L^2(\R^d)$ and almost everywhere, and $\lambda_k:=\lambda_{\varepsilon_k} \to \lambda$ as
$k\to\infty$.
\end{proposition}
\begin{proof}
Let $\varepsilon_k \downarrow 0$ be any subsequence and set $V_k:=V_{\varepsilon_k}$. Since $\sup_{k \ge 0}\Norm{V_k}
{\infty}=v<\infty$, we have $V_k \to V$ in any $L^p(\R^d)$ for $1 \le p <\infty$ and all $V_k$ and $V$ satisfy Assumption
\ref{ass:existence}. Then we obtain the statement by Theorem \ref{thm:stability} and taking a suitable subsequence.
\end{proof}
From now on, we denote by $\varphi_k$ the sequence identified by this result.
\begin{theorem}
\label{thm:regularphi0}
We have $\Phi_{m,\alpha}(-\Delta)\varphi \in L^{2}(\R^d)$. Furthermore, $\Phi_{m,\alpha}(-\Delta)\varphi_k \to
\Phi_{m,\alpha}(-\Delta)\varphi$ in $L^2(\R^d)$, as $k\to\infty$.
\end{theorem}
\begin{proof}
By a similar argument as in Proposition \ref{prop:regular}, we see that
\begin{equation*}
\Phi_{m,\alpha}(|\xi|^2)\widehat{\varphi}(\xi)=\lambda\widehat{\varphi}(\xi)+\cF[V\varphi](\xi),
\end{equation*}
for almost every $\xi \in \R^d$, hence $\Phi_{m,\alpha}(|\xi|^2)\widehat{\varphi}(\xi) \in L^2(\R^d)$, implying
that $\Phi_{m,\alpha}(-\Delta)\varphi \in L^2(\R^d)$. Note that
\begin{align*}
\int_{\R^d}&|V_k(x)\varphi_{k}(x)-V(x)\varphi(x)|^2dx
\le 2\int_{\R^d}|V_k(x)|^2 |\varphi_{k}(x)-\varphi(x)|^2dx\\
&\qquad +2\int_{\R^d}|V(x)-V_k(x)|^2|\varphi(x)|^2dx\\
&\le 2v^2\Norm{\varphi_{k}-\varphi}{2}^2 + 2\Norm{\varphi}{\infty}^2\Norm{V_k-V}{2} \to 0,
\end{align*}
where we used the fact that $\varphi \in C(\R^d)\cap L^\infty(\R^d)$, as specified in Remark \ref{rmk:continuity}.
Hence $V_k\varphi_k\to V\varphi$ in $L^2(\R^d)$. Clearly, this implies $\cF[V_k\varphi_k] \to \cF[V\varphi]$
in $L^2(\R^d)$, and thus we get $\Phi_{m,\alpha}(|\cdot|^2) \widehat{\varphi}_k \to \Phi_{m,\alpha}(|\cdot|^2)
\widehat{\varphi}$, which in turn implies that $\Phi_{m,\alpha}(-\Delta)\varphi_k \to \Phi_{m,\alpha}(-\Delta)\varphi$
in $L^2(\R^d)$.
\end{proof}

\subsection{Radial decrease of the ground state}

The fact that the ground states are unique and the potentials are rotationally symmetric give us some information about
the shape of $\varphi$.
\begin{proposition}
\label{rotsym}
Let $\Phi \in \cB_0$ and $V \in L^{\infty,0}(\R^d)$ be rotationally symmetric and satisfy Assumption \ref{ass:existence}.
Then the ground state $\varphi$ of $\cH_{\Phi,V}$ is rotationally symmetric.
\end{proposition}
\begin{proof}
Consider an arbitrary rotation ${\rm R} \in \rm{SO}(d)$ and let $\widetilde{\varphi}(x)=\varphi({\rm R}x)$.
Using that $\Norm{\widetilde{\varphi}}{2}=1$ and $\widetilde{\varphi}(x)> 0$ for all $x \in \R^d$, we obtain
\begin{align*}
\cA_{\Phi,V}(\widetilde{\varphi},\widetilde{\varphi})&=\int_{\R^d}\int_{\R^d}|\widetilde{\varphi}(x)-
\widetilde{\varphi}(y)|j_{\Phi}(|x-y|)dxdy+\int_{\R^d}V_k(x)\widetilde{\varphi}^2(x)dx\\
&=\int_{\R^d}\int_{\R^d}|\varphi({\rm R}x)-\varphi({\rm R}y)|j_{\Phi}(|x-y|)dxdy+\int_{\R^d}V_k(x)
\varphi({\rm R}x)dx\\
&=\int_{\R^d}\int_{\R^d}|\varphi({\rm R}x)-\varphi({\rm R}y)|j_{\Phi}(|{\rm R}x-{\rm R}y|)dxdy
+\int_{\R^d}V_k({\rm R}x)\varphi({\rm R}x)dx\\
&=\int_{\R^d}\int_{\R^d}|\varphi(x)-\varphi(y)|j_{\Phi}(|x-y|)dxdy+\int_{\R^d}V_k(x)\varphi(x)dx
=\lambda.
\end{align*}
Hence $\widetilde{\varphi}$ is an eigenfunction of $\cH_{\Phi,V}$ at the same eigenvalue $\lambda$. By
uniqueness we then have $\widetilde{\varphi}=\varphi$, thus $\varphi$ is invariant to all rotations.
\end{proof}

In the following we will need to work with antisymmetric functions. We will denote any $x \in \R^d$ as $x=(x_1,x')$,
where $x_1 \in \R$ and $x' \in \R^{d-1}$. For every $\mu \in \R$ we write $x^\mu=(2\mu-x_1,x')$ and $\cU_\mu =
\{x \in \R^d: \ x_1<\mu\}$. We say that a function $w$ is $\mu$-antisymmetric if  $w(x^\mu)=-w(x)$ for every $x
\in \cU_\mu$. Recall the following result from \cite[Lem. 2.1]{CHL17} for the fractional Laplacian.
\begin{lemma}
\label{lem:fraclapcont}
Let $w \in C^2(\R^d) \cap L^\infty(\R^d)$ be a $\mu$-antisymmetric function, and suppose that there exists
$x \in \cU_\mu$ such that $w(x)=\inf_{y \in \cU_\mu}w(y)<0$. Denote $\delta=\mu-x_1$. Then there exists a constant
$C = C_{d,\alpha} > 0$ dependent only on $d$ and $\alpha$ such that
\begin{equation}\label{eq:fraclapcont}
\Phi_{0,\alpha}(-\Delta)w(x)
\le
C_{d,\alpha}\left(\delta^{-\alpha}w(x)-\delta \int_{\cU_\mu} \frac{(w(y)-w(x))(\mu-y_1)}{|x-y^\mu|^{d+\alpha+2}}dy\right).
\end{equation}
\end{lemma}
We will extend this result to the operator $\Phi_{m,\alpha}(-\Delta)$. To do this, for every $m, r >0$ we define
\begin{align*}
\label{sigma}
\begin{split}
\sigma_{m,\alpha}(r)
&=
\frac{\alpha 2^{1-\frac{d-\alpha}{2}}}{\Gamma\left(1-\frac{\alpha}{2}\right)
\pi^{d/2}} \left(\frac{2^{\frac{d+\alpha}{2}-1}\Gamma\left(\frac{d+\alpha}{2}\right)}{r^{d+\alpha}}-
\frac{m^{\frac{d+\alpha}{2\alpha}}K_{\frac{d+\alpha}{2}}\left(m^{1/\alpha}r\right)}{r^{\frac{d+\alpha}{2}}}\right)\\
&=
\frac{\alpha 2^{1-\frac{d-\alpha}{2}}}{\Gamma\left(1-\frac{\alpha}{2}\right)\pi^{d/2}}{\frac{1}{r^{d+\alpha}}}
\int_0^{m^{1/\alpha} r}w^{\frac{d+\alpha}{2}}K_{\frac{d+\alpha}{2}-1}(w)dw,
\end{split}
\end{align*}
with the same Bessel function as used before. It has been shown in \cite[Lem. 2]{R02} that $\int_{\R^d}\sigma_{m,\alpha}(|x|)dx
=m$ and the decomposition
\begin{equation}
\label{dec}
j_{0,\alpha}(r)=j_{m,\alpha}(r)+\sigma_{m,\alpha}(r)
\end{equation}
holds. Due to this observation, we can define the operator $G_{m,\alpha}: L^\infty(\R^d) \to L^{\infty}(\R^d)$ given by
\begin{equation*}
G_{m,\alpha}f(x)=\frac{1}{2}\int_{\R^d}(f(x+h)-2f(x)+f(x-h))\sigma_{m,\alpha}(|h|)dh.
\end{equation*}
In particular, for every $f \in C_{\rm c}^\infty(\R^d)$ we have
\begin{equation}
\label{withG}
\Phi_{m,\alpha}(-\Delta)f=\Phi_{0,\alpha}(-\Delta)f-G_{m,\alpha}f.
\end{equation}
Clearly, this relation can be extended to any function $f$ such that both $\Phi_{m,\alpha}(-\Delta)f$ and
$\Phi_{0,\alpha}(-\Delta)f$ are defined pointwise via \eqref{eq:Phidelta}. For the details and proof of such a
decomposition, we refer to \cite[Sect. 2.3.2]{AL}.

Before proceeding with the extension, we determine the derivative of the jump kernel $j_{m,\alpha}$.
\begin{lemma}
\label{lem:nuprime}
We have
\begin{equation*}
j_{m,\alpha}'(r)=-\frac{\alpha 2^{\frac{\alpha-d}{2}}m^{\frac{d+\alpha+2}{2\alpha}}}
{\pi^{d/2}\Gamma\left(1-\frac{\alpha}{2}\right)}
\frac{K_{(d+\alpha +2)/2}\left(m^{1/\alpha}r\right)}{r^{\frac{d+\alpha}{2}}}, \quad r>0.
\end{equation*}
\end{lemma}
\begin{proof}
Using the formula
\begin{equation*}
\frac{d\, }{d\, r}K_{\xi}(r)=-\frac{K_{\xi+1}(r)+K_{\xi-1}(r)}{2},
\end{equation*}
see \cite[\S 3.71(4)]{Wat}, we obtain
\begin{align}
\label{eq:jprime}
		\begin{split}
			j_{m,\alpha}'(r)
			&=
			-\frac{\alpha 2^{\frac{\alpha-d}{2}}m^{\frac{d+\alpha}{2\alpha}}}{2\pi^{d/2}\Gamma\left(1-\frac{\alpha}{2}\right)}
			r^{-\frac{d+\alpha+2}{2}}\left((d+\alpha)K_{(d+\alpha)/2}\big(m^{1/\alpha}r\big)\right.\\
			&\left. \qquad
			+m^{1/\alpha}rK_{(d+\alpha-2)/2}\big(m^{1/\alpha}r\big)+m^{1/\alpha}rK_{(d+\alpha +2)/2}
			\big(m^{1/\alpha}r\big)\right).
		\end{split}
	\end{align}
Since $2\xi K_\xi(r)+rK_{\xi-1}(r)=rK_{\xi+1}(r)$, see \cite[\S 3.71(1)]{Wat}, a combination with \eqref{eq:jprime} shows
the claim.
\end{proof}
\begin{lemma}
\label{lem:rellapcont}
Let $m>0$, $w \in C^2(\R^d)\cap L^\infty(\R^d)$ be a $\mu$-antisymmetric function, and suppose that there exists $x \in \cU_\mu$
such that
\begin{equation*}
w(x)=\inf_{y \in \cU_\mu}w(y)<0.
\end{equation*}
Denote $\delta=\mu-x_1$. Then there exists a constant $C = C(m,\alpha,d) > 0$ such that
\begin{equation}
\label{eq:rellapcont1}
\Phi_{m,\alpha}(-\Delta)w(x)\le C\left((\delta^{-\alpha}-m)w(x)-\delta \int_{\cU_\mu}\frac{(w(y)-w(x))(\lambda-y_1)
K_{(d+\alpha +2)/2}(m^{1/\alpha}|x-y^\mu|)}{|x-y^\mu|^\frac{d+\alpha+2}{2}}dy\right).
\end{equation}
Moreover, if $\delta \in (0,\delta_1]$, with $0<\delta_1<\infty$, then there exists $C = C(m,\alpha,d,\delta_1) >0$
such that	
\begin{equation}
\label{eq:rellapcont2}
\Phi_{m,\alpha}(-\Delta)w(x)\le C\left(\delta^{\frac{d-\alpha}{2}}w(x)
-\delta \int_{\cU_\mu}\frac{(w(y)-w(x))(\lambda-y_1)
K_{(d+\alpha +2)/2}(m^{1/\alpha}|x-y^\mu|)}{|x-y^\mu|^\frac{d+\alpha+2}{2}}dy\right).
\end{equation}
\end{lemma}
\begin{proof}
We proceed through several steps.

\medskip
\noindent
\emph{Step 1:} First observe that since $w \in C^2(\R^d)$, then $w(y)=0$ for every $y \in \partial \cU_\mu$.
Hence $d(x,\partial \cU_\mu)=\delta>0$. Since the expression \eqref{eq:Phidelta} applies for $w(x)$, we have
\begin{equation*}
\Phi_{m,\alpha}(-\Delta)w(x)=\lim_{\varepsilon \downarrow 0}\int_{\R^d \setminus B_\varepsilon(x)}(w(x)-w(y))j_{m,\alpha}(|x-y|)dy.
\end{equation*}
Fix $\varepsilon>0$ and choose it small enough to have $B_\varepsilon(x)\subset \cU_\mu$, which can be done since $\delta>0$. We
split up the integral in two parts as
\begin{align*}
\int_{\R^d \setminus B_\varepsilon(x)}&(w(x)-w(y))j_{m,\alpha}(|x-y|)dy =
\left(\int_{\cU_\mu \setminus B_\varepsilon(x)}+\int_{(\R^d \setminus \cU_\mu)\setminus B_\varepsilon(x)}\right)(w(x)-w(y))j_{m,\alpha}(|x-y|)dy
\end{align*}
and apply the change of variable $y \mapsto y^\mu$ which, using that $w$ is $\mu$-antisymmetric, gives
\begin{align*}
\int_{\R^d \setminus B_\varepsilon(x)}&(w(x)-w(y))j_{m,\alpha}(|x-y|)dy =
\int_{\cU_\mu \setminus B_\varepsilon(x)}(w(x)-w(y))j_{m,\alpha}(|x-y|)dy\\
&\qquad + \int_{\cU_\mu}(w(x)+w(y))j_{m,\alpha}(|x-y^\mu|)dy\\
&=
-\int_{\cU_\mu \setminus B_\varepsilon(x)}(w(y)-w(x))(j_{m,\alpha}(|x-y|)-j_{m,\alpha}(|x-y^\mu|))dy\\
&\qquad + 2w(x)\int_{\cU_\mu \setminus B_\varepsilon(x)}j_{m,\alpha}(|x-y^\mu|)dy+
\int_{B_\varepsilon(x)}(w(x)+w(y))j_{m,\alpha}(|x-y^\mu|)dy.
\end{align*}
Since $w \in L^\infty(\R^d)$, we can take the limit $\varepsilon \downarrow 0$ giving
\begin{align}
\label{eq:intest02}
\begin{split}
\Phi_{m,\alpha}(-\Delta)w(x)
			&=
			-\lim_{\varepsilon \downarrow 0}\int_{\cU_\mu \setminus \cB_\varepsilon(x)}(w(y)-w(x))
			(j_{m,\alpha}(|x-y|)-j_{m,\alpha}(|x-y^\mu|))dy\\
			& \qquad +
			2w(x)\int_{\cU_\mu}j_{m,\alpha}(|x-y^\mu|)dy.
		\end{split}
	\end{align}
	By the decomposition formulae \eqref{dec}-\eqref{withG} we have
	\begin{align}\label{eq:intest01}
		\begin{split}
			L_{m,\alpha}w(x)
			&=
			-\lim_{\varepsilon \downarrow 0}\int_{\cU_\mu \setminus \cB_\varepsilon(x)}(w(y)-w(x))
			(j_{m,\alpha}(|x-y|)-j_{m,\alpha}(|x-y^\mu|))dy\\
			& \qquad
			+C_1(d,\alpha) 2w(x)\int_{\cU_\mu}\frac{1}{|x-y^\mu|^{d+\alpha}}dy-2w(x)\int_{\cU_\mu}\sigma_m(|x-y^\mu|)dy,
		\end{split}
	\end{align}
	where we denote $C_1(d,\alpha)=\frac{\alpha 2^\frac{\alpha-d}{2}}{\pi^{d/2}\Gamma\left(1-\frac{\alpha}{2}\right)}$.
	
	\medskip
	\noindent
	\emph{Step 2:}
	Next we estimate the integrals one by one, starting with the third. We have
	\begin{equation*}
		\int_{\cU_\mu}\sigma_{m,\alpha}(|x-y^\mu|)dy
		=\int_{-\infty}^{\mu}\int_{\R^{d-1}}\sigma_{m,\alpha}((|x'-y'|^2-|2\mu-x_1-y_1|^2)^{1/2})dy'dy_1.
	\end{equation*}
	Setting $z'=x'-y'$ and $z_1=2\mu-x_1-y_1$ we have
	\begin{equation*}
		\int_{\cU_\mu}\sigma_{m,\alpha}(|x-y^\mu|)dy=\int_{\mu-x_1}^{\infty}
		\int_{\R^{d-1}}\sigma_{m,\alpha}((|z'|^2-|z_1|^2)^{1/2})dz'dz_1.
	\end{equation*}
	With $\widetilde{\cU}_0=\{x \in \R^d: \ x_1>0\}$, we have
	\begin{equation*}\label{eq:intest1}
		\int_{\cU_\mu}\sigma_{m,\alpha}(|x-y^\mu|)dy\le \int_{\widetilde{\cU}_0}\sigma_{m,\alpha}(|z|)dz=\frac{m}{2}.
	\end{equation*}

		Next consider the second integral. We have
	\begin{align*}
		\int_{\cU_\mu}\frac{dy}{|x-y^\mu|^{d+\alpha}}
		&=
		\int_{-\infty}^{\mu}\int_{\R^{d-1}}\frac{dy'dy_1}{(|x'-y'|^2+|2\mu-x_1-y_1|^2)^{\frac{d+\alpha}{2}}}\\
		&=
		\frac{1}{(\mu-x_1)^{d+\alpha}}\int_{-\infty}^{\mu}\int_{\R^{d-1}}\frac{dy'dy_1}{\Big(\big|\frac{x'-y'}{\mu-x_1}\big|^2
			+\big|1+\frac{\mu-y_1}{\mu-x_1}\big|^2\Big)^{\frac{d+\alpha}{2}}}.
	\end{align*}
	Setting $z'=\frac{x'-y'}{\mu-x_1}$ and $z_1=\frac{\mu-y_1}{\mu-x_1}$ we get
	\begin{align}
\nonumber
		\begin{split} \int_{\cU_\mu}\frac{dy}{|x-y^\mu|^{d+\alpha}}
			&=
			\frac{1}{(\mu-x_1)^\alpha}\int_{0}^{\infty}\int_{\R^{d-1}}\frac{dz'dz_1}{(|z'|^2+|1+z_1|^2)^{\frac{d+\alpha}{2}}}
			=
			C_2(d,\alpha)\delta^{-\alpha},
		\end{split}
	\end{align}
	with constant
	$C_2(d,\alpha)=\int_{0}^{\infty}\int_{\R^{d-1}}(|z'|^2+|1+z_1|^2)^{-(d+\alpha)/2}dz'dz_1 < \infty$.
	
	Finally consider the first integral. By Lagrange's theorem there exists $|x-y| < \theta(x,y) < |x-y^\mu|$ such
	that
	\begin{align*}
		\int_{\cU_\mu \setminus B_\varepsilon(x)}&(w(y)-w(x))\big(j_{m,\alpha}(|x-y|)-j_{m,\alpha}(|x-y^\mu|)\big)dy\\
		&=
		\int_{\cU_\mu \setminus B_\varepsilon(x)}(w(y)-w(x))j'_{m,\alpha}(\theta(x,y))(|x-y|-|x-y^\mu|)dy\\
		&=
		C_3(d,m,\alpha)\int_{\cU_\mu \setminus B_\varepsilon(x)}(w(y)-w(x))
		\frac{K_{(d+\alpha +2)/2}(m^{1/\alpha}\theta(x,y))}{\theta(x,y)^{\frac{d+\alpha+2}{2}}}(|x-y|-|x-y^\mu|)dy,
	\end{align*}
	where we used Lemma \ref{lem:nuprime} and denoted $C_3(d,m,\alpha)=\frac{\alpha 2^{\frac{\alpha-d}{2}}
		m^{\frac{d+\alpha+2}{2\alpha}}}{\pi^{d/2}\Gamma\left(1-\frac{\alpha}{2}\right)}$. With the choice of $x$,
	recall that $w(y)-w(x)\ge 0$ as $y \in \cU_{\mu}\setminus B_\varepsilon(x)$, which yields
	\begin{align*}
		C_3(d,m,\alpha)\int_{\cU_\mu \setminus B_\varepsilon(x)}&(w(y)-w(x))
		\frac{K_{(d+\alpha +2)/2}(m^{{1/\alpha}}\theta)}{\theta^{\frac{d+\alpha}{2}}}(|x-y^\mu|-|x-y|)dy\\
		&\ge
		C_3(d,m,\alpha)\int_{\cU_\mu \setminus B_\varepsilon(x)}(w(y)-w(x))
		\frac{K_{(d+\alpha +2)/2}(m^{{1/\alpha}}|x-y^\mu|)}{|x-y^\mu|^{\frac{d+\alpha}{2}}}|y^\mu-y|dy\\
		&=
		C_3(d,m,\alpha)\int_{\cU_\mu \setminus B_\varepsilon(x)}(w(y)-w(x))
		\frac{K_{(d+\alpha +2)/2}(m^{{1/\alpha}}|x-y^\mu|)}{|x-y^\mu|^{\frac{d+\alpha}{2}}}(\mu-y_1)dy.
	\end{align*}
	Observe that $|x-y^\mu|\ge \inf_{z \in \widetilde{\cU}_\mu}|x-z|=\mu-x_1=\delta$, and thus
	\begin{align}\label{eq:intest3}
		\begin{split}
			\int_{\cU_\mu \setminus B_\varepsilon(x)}&(w(y)-w(x))(j_{m,\alpha}(|x-y|)-j_{m,\alpha}(|x-y^\mu|))dy\\&
			\ge
			2C_3(d,m,\alpha)\int_{\cU_\mu \setminus B_\varepsilon(x)}(w(y)-w(x))
			\frac{K_{(d+\alpha +2)/2}(m^{{1/\alpha}}|x-y^\mu|)}{|x-y^\mu|^{\frac{d+\alpha}{2}}}(\mu-y_1)dy\\
			&\ge
			C_3(d,m,\alpha)\delta\int_{\cU_\mu \setminus B_\varepsilon(x)}(w(y)-w(x))\frac{K_{(d+\alpha +2)/2}
				(m^{{1/\alpha}}|x-y^\mu|)}{|x-y^\mu|^{\frac{d+\alpha+2}{2}}}(\mu-y_1)dy.
		\end{split}
	\end{align}
	Combining \eqref{eq:intest01}-\eqref{eq:intest3}
	and writing $C=C(d,m,\alpha)=\min\big\{2C_1(d,\alpha),C_2(d,\alpha),2C_3(d,m,\alpha)\big\}$ we
	obtain \eqref{eq:rellapcont1}.
	
	\medskip
	\noindent
	\emph{Step 3:}
	To complete the proof, consider the second integral in \eqref{eq:intest02}. Using the explicit form of
	$j_{m,\alpha}$, we obtain the expressions
	\begin{align*}
		\int_{\cU_\mu}\frac{K_{(d+\alpha)/2}(m^{{1/\alpha}}|x-y^\mu|)}{|x-y^\mu|^{\frac{d+\alpha}{2}}}dy
		&=
		\int_{-\infty}^{\mu}\int_{\R^{d-1}}\frac{K_{(d+\alpha)/2}(m^{{1/\alpha}}
			(|x'-y'|^2+|2\mu-x_1-y_1|^2)^{1/2})}{(|x'-y'|^2+|2\mu-x_1-y_1|^2)^{\frac{d+\alpha}{4}}}dy'dy_1\\
		&=
		\frac{1}{\delta^{\frac{d+\alpha}{2}}}\int_{-\infty}^{\mu}\int_{\R^{d-1}}\frac{K_{(d+\alpha)/2}
			(m^{{1/\alpha}}\delta(|\frac{x'-y'}{\mu-x_1}|^2+|1+
			\frac{\mu-y_1}{\mu-x_1}|^2)^{1/2})}{(|\frac{x'-y'}{\mu-x_1}|^2+|1
			+\frac{\mu-y_1}{\mu-x_1}|^2)^{\frac{d+\alpha}{4}}}dy'dy_1 \\
		&=
		\delta^{\frac{d-\alpha}{2}}\int_{0}^{\infty}\int_{\R^{d-1}}\frac{K_{(d+\alpha)/2}(m^{{1/\alpha}}
			\delta(|z'|^2+|1+z_1|^2)^{1/2})}{(|z'|^2+|1+z_1|^2)^{\frac{d+\alpha}{4}}}dz'dz_1,
	\end{align*}
	where the last line is obtained by setting $z'=\frac{x'-y'}{\mu-x_1}$ and $z_1=\frac{\mu-y_1}{\mu-x_1}$.
	Since $K_{(d+\alpha)/2}$ is decreasing and $\delta \in (0,\delta_1]$, we have
	\begin{align}\label{eq:intest4}
		\begin{split}
			\int_{\cU_\mu}\frac{K_{(d+\alpha)/2}(m^{{1/\alpha}}|x-y^\mu|)}{|x-y^\mu|^{\frac{d+\alpha}{2}}}dy
			\ge C_4(d,m,\alpha,\delta_1)\delta^{\frac{d-\alpha}{2}},
		\end{split}
	\end{align}
	with
	\begin{equation*}
		C_4(d,m,\alpha,\delta_1)
		=\int_{0}^{\infty}\int_{\R^{d-1}}\frac{K_{(d+\alpha)/2}(m^{{1/\alpha}}
			\delta_1(|z'|^2+|1+z_1|^2)^{1/2})}{(|z'|^2+|1+z_1|^2)^{\frac{d+\alpha}{4}}}dz'dz_1<\infty.
	\end{equation*}
	Applying \eqref{eq:intest3}-\eqref{eq:intest4} to \eqref{eq:intest02} (recall that $w(x)<0$), we arrive at
	\eqref{eq:rellapcont2} with the constant $C=C(d,m,\alpha,\delta_1)=\min\{2C_3(d,m,\alpha),C_4(d,m,\alpha,\delta_1)\}$.
\end{proof}

\begin{remark}
\rm{
Since $K_\xi(|x|) \simeq 2^{\xi-1}\Gamma(\xi) |x|^{-\xi}$ as $|x| \downarrow 0$, we see that, apart possibly from the
numerical prefactor, estimate \eqref{eq:rellapcont1} gives consistently back the estimate \eqref{eq:fraclapcont} in the
limit $m \downarrow 0$.
}
\end{remark}
The following is a consequence of Lemmas \ref{lem:fraclapcont} and \ref{lem:rellapcont}, which we will use below.
\begin{corollary}
\label{cor:negLap}
Let $m \ge 0$, $w \in C^2(\R^d) \cap L^\infty(\R^d)$ be a $\mu$-antisymmetric function, and suppose that there exists
$x \in \cU_\mu$ such that $w(x)=\inf_{y \in \cU_\mu}w(y)<0$. Denote $\delta=\mu-x_1$. Then
\begin{equation*}
\Phi_{m,\alpha}(-\Delta)w(x)<0.
\end{equation*}
\end{corollary}
Now we are ready to prove the following monotonicity result for $\varphi_{k}$.
\begin{theorem}
\label{prop:mono1}
Let $\chi_{k}:[0,\infty) \to \R$ be such that $\varphi_{k}(x)=\chi_{k}(|x|)$. Then $\chi_{k}$ is non-increasing
in $[a+\varepsilon_k,\infty)$.
\end{theorem}
\begin{proof}
For any $\mu \le 0$, let $\varphi_{k}^\mu(x)=\varphi_{k}(x^\mu)$ and $w_\mu^k(x):=\varphi_{k}^\mu(x)-\varphi_{k}(x)$.
Clearly, since $\varphi_{k} \in C^2(\R^d)\cap L^\infty(\R^d)$, also $w_{\mu}^k \in C^2(\R^d) \cap L^\infty(\R^d)$, for every
$\mu \le 0$. Furthermore, it is $\mu$-antisymmetric by construction. Also, $\varphi_{k}(x) \to 0$ as $|x| \to \infty$ (see,
\cite[Cor. 4.1, 4.3]{KL17}).

Let $\mu \le -(a+\varepsilon_k)$ and assume that $\inf_{x \in \cU_\mu}w_{\mu}^k(x)<0$. Since $w_{\mu}^k(x) \to 0$ as $|x| \to
0$ and $w_\mu^k$ is continuous, the infimum is actually a minimum, i.e., there exists $x_* \in \overline{\cU}_\mu$ such that
$w_\mu^k(x_*)=\inf_{x \in \cU_\mu}w_{\mu}^k(x)$. Obviously, $x_* \not \in \partial \cU_\mu$, otherwise $w_\mu^k(x_*)$ and then
$\Phi_{m,\alpha}(-\Delta)w_\mu^k(x_*)<0$ by Corollary \ref{cor:negLap}. On the other hand,
\begin{align}
\label{eq:oponanti}
\begin{split}
\Phi_{m,\alpha}(-\Delta)w_\mu^k(x_*)
&=\Phi_{m,\alpha}(-\Delta)\varphi_{k}^\mu(x_*)-\Phi_{m,\alpha}(-\Delta)\varphi_{k}(x_*)\\
&=\Phi_{m,\alpha}(-\Delta)\varphi_{k}(x_*^\mu)-\Phi_{m,\alpha}(-\Delta)\varphi_{k}(x_*)\\
&=(\lambda_k-V_k(x_*^\mu))\varphi_{k}(x_*^\mu)-(\lambda_k-V_k(x_*))\varphi_{k}(x_*)\\
&=(\lambda_k-V_k(x_*)) w_\mu^k(x_*)+(V_k(x_*)-V_k(x_*^\mu))\varphi_{k}(x_*).
\end{split}
\end{align}
Since $x_* \in \cU_{\mu}$, we have $|x_*|\ge |x_1|>a+\varepsilon_k$ and $V_k(x_*)=0$. Hence we get
\begin{align*}
\Phi_{m,\alpha}(-\Delta)w_\mu^k(x_*)&=\lambda w_\mu^k(x_*)-V_k(x_*^\mu)\varphi_{k}(x_*)>0,
\end{align*}
which is a contradiction. Hence $\inf_{x \in \cU_\mu}w_\mu^k(x)\ge 0$ and thus $w_\mu^k(x)\ge 0$ for all $x \in \cU_\mu$.

Now let $r_1>r_2 \ge a+\varepsilon_k$ and consider $x=-r_1\mathbf{e}_1$ and $y=-r_2\mathbf{e}_1$, where $\mathbf{e}_1=
(1,0,\dots,0)$. Furthermore, let $\mu=-\frac{r_1+r_2}{2}$. Clearly, $\mu \le -(a+\varepsilon)$, $x \in \cU_\mu$ and $y=x^\mu$.
The previous argument guarantees that $w_\mu(x)\ge 0$, which implies that $\chi_{k}(r_2)=\varphi_{k}(y)\ge \varphi_{k}(x)
=\chi_{k}(r_1)$.
\end{proof}
As a consequence of this and Theorem \ref{thm:stability}, we get the following monotonicity result for $\varphi$.
\begin{theorem}
\label{prop:mono2}
Let $\chi: [0,\infty) \to \R$ be such that $\varphi(x)=\chi(|x|)$. Then $\chi$ is non-increasing on $[a,\infty)$.
\end{theorem}
\begin{proof}
Let $\Omega=\{x \in \R^d: \ \lim_{k \to \infty}\varphi_{k}(x)=\varphi(x)\}$ and observe that $|\R^d \setminus \Omega|=0$.
As a consequence, $\Omega$ is dense in $\R^d$. Let $r_1>r_2>a$ and consider $x=-r_1\mathbf{e}_1$ and $y=-r_2 \mathbf{e}_2$.
Consider two sequences $(x^\ell)_{\ell \in \N}, (y^\ell)_{\ell \in \N} \subset \Omega$ with $x^\ell \to x$ and $y^\ell \to y$.
Since $|x|>|y|>a$, we can assume without loss of generality that $|x^\ell|>|y^\ell|>a$ for every $\ell \in \N$. Now fix $\ell
\in \N$ and observe that there exists $k_* \in \N$ such that $|y^\ell|>a+\varepsilon_k$ for every $k \ge k_*$. By Theorem
\ref{prop:mono1} we know that $\varphi_{k}(x^\ell) \le \varphi_{k}(y^\ell)$ for every $k \ge k_*$. Taking the limit as $k
\to \infty$ and using the fact that $x^\ell,y^\ell \in \Omega$, we obtain $\varphi(x^\ell)\le \varphi(y^\ell)$. Since
$\varphi$ is continuous, taking the limit $\ell \to \infty$ we get $\varphi(x)\le \varphi(y)$. Finally, the case
$r_2=a$ is obtained by continuity of $\varphi$.
\end{proof}

Next we prove that $\chi$ is non-increasing also in $[0,a]$. To do this, we need a family of auxiliary functions. Specifically,
for every $\mu \le 0$ define $w_\mu(x)=\varphi(x^\mu)-\varphi(x)$.
\begin{lemma}
\label{lem:mono3}
Let $\mu \in (-a,0]$. If $x \in \cU_\mu$ is such that $x^\mu \not \in B_a$, then $w_\mu(x)\ge 0$.
\end{lemma}
\begin{proof}
Let $x \in \R^d$ be such that $x^\mu \not \in B_a$. Since $a<\mu<0$ we have that $x \not \in B_a$. Note that $x=(x_1,x')$ and
$x^\mu=(2\mu-x_1,x')$. Since $x_1<\mu<0$, it follows that $|2\mu-x_1|\le |\mu|+|\mu-x_1|=-\mu+\mu-x_1=-x_1=|x_1|$. Hence $|x|
\ge |x^\mu|\ge a$ and $\varphi(x)\le \varphi(x^\mu)$ by Proposition \ref{prop:mono2}, which shows the claim.
\end{proof}

The following result highlights the relation between the ground state eigenvalue and the principal Dirichlet eigenvalue of
the support of the potential well.
\begin{lemma}
\label{lem:mono4}
Let $\lambda_a$ be the principal Dirichlet eigenvalue of $B_a$. Then $\lambda+v<\lambda_a$.
\end{lemma}
\begin{proof}
Let $f_a$ be a Dirichlet eigenfunction with $\Norm{f_a}{2}=1$. Since $\varphi$ is the minimizer of $\cA_{m,\alpha}:=
\cA_{\Phi_{m,\alpha},v}$ on the set $\{u \in H^{\alpha/2}(\R^d): \Norm{u}{2}=1\}$, we have
\begin{equation*}
\lambda=\cA_{m,\alpha}(\varphi,\varphi)<\cA_{m,\alpha}(f_a,f_a),
\end{equation*}
where the inequality is strict by the fact that $\lambda$ is a simple eigenvalue. Since $f_a$ is supported in $B_a$, we have
\begin{equation*}
\cA_{m,\alpha}(f_a,f_a)=\cE_{m,\alpha}(f_a,f_a)+\int_{B_a}V(x)f^2_a(x)dx=\lambda_a-v,
\end{equation*}
which shows the claim.
\end{proof}

Now we are finally ready to prove that $\varphi_0$ is radially decreasing.
\begin{theorem}
\label{monout}
Let $\chi: [0,\infty) \to \R$ be such that $\varphi(x)=\chi(|x|)$. Then $\chi$ is non-increasing.
\end{theorem}
\begin{proof}
As Theorem \ref{prop:mono2} guarantees that $\chi$ is non-increasing in $[a,\infty)$, we only need to prove that it is also
non-increasing in $[0,a]$. To do this, let $\mu \in (-a,0]$ and consider $w_\mu$. Let also $w^k_\mu$ as in the proof of Theorem
\ref{prop:mono1} and observe that by Theorem \ref{thm:regularphi0} we have $\Phi_{m,\alpha}(-\Delta)w_\mu^k \to \Phi_{m,\alpha}
(-\Delta)w_\mu$ in $L^2(\R^d)$. Consider an arbitrary function $u \in C_{\rm c}^\infty(\R^d)$ and note that
\begin{equation}
\label{eq:weakapprox}
\langle \Phi_{m,\alpha}(-\Delta)w_\mu^k, u\rangle=\langle (\lambda_k-V_k)w_\mu^k, u\rangle
+ \langle (V_k-V_k^\mu)\varphi_{k}, u\rangle,
\end{equation}
where $V_k^\mu(x)=V_k(x^\mu)$. Recall also that $w_\mu^k \to w_\mu$ in $L^2(\R^d)$ and $\Norm{w_\mu}{\infty}\le
2 \Norm{\varphi}{\infty}<\infty$. Thus we get
\begin{align*}
\int_{\R^d}&|(\lambda_k-V_k(x))w_\mu^k(x)-(\lambda-V(x))w_\mu(x)|^2dx
\le 2\int_{\R^d}|\lambda_k w_\mu^k(x)-\lambda w_\mu(x)|^2dx\\
&\quad +2\int_{\R^d}|V_k(x)w_\mu^k(x)-V(x)w_\mu(x)|^2dx\\
&\le 2(|\lambda_k|^2+v^2)\Norm{w_\mu^k-w_\mu}{2}^2+2\Norm{w_\mu}{2}^2|\lambda_k-\lambda|^2
+4\Norm{\varphi}{\infty}^2\Norm{V_k-V}{2}^2,
\end{align*}
which implies $(\lambda_k-V_k)w_\mu^k \to (\lambda-V)w_\mu$ in $L^2(\R^d)$. Arguing in the same way, we get
$(V_k-V_k^\mu)\varphi_{k} \to (V-V^\mu)\varphi$ in $L^2(\R^d)$, where $V^\mu(x)=V(x^\mu)$. Hence by taking the limit
$k \to \infty$ in \eqref{eq:weakapprox}, we obtain
\begin{equation}
\label{eq:weakeigen}
\cE_{m,\alpha}(w_\mu,u)=\langle \Phi_{m,\alpha}(-\Delta)w_\mu, u\rangle=\langle (\lambda-V)w_\mu, u\rangle
+ \langle (V-V^\mu)\varphi, u\rangle.
\end{equation}
Since $u \in C_{\rm c}^\infty(\R^d)$ is arbitrary, equality \eqref{eq:weakeigen} holds also for $u \in H^{\alpha/2}(\R^d)$.
At this point we borrow an argument from \cite[Prop. 3.1]{FJ15}. Since $w_\mu$ is continuous and thus $w_\mu(x)=0$ for all
$x \in \partial \cU_\mu$, we may consider $u=w_\mu^-\mathbf{1}_{\cU_\mu}$ as a test function (see, e.g., \cite[Exercise 3.22]
{M00}), where $w_\mu^-=\max\{-w_\mu,0\}$, and observe that
\begin{align*}
(w_\mu(x)-w_\mu(y))(u(x)-u(y))=-(u(x)-u(y))^2-u(x)(w_\mu(y)+u(y))-u(y)(w_\mu(x)-u(x)).
\end{align*}
Hence we have
\begin{align*}
\cE_{m,\alpha}&(w_\mu,u)=-\cE_{m,\alpha}(u,u)-\int_{\R^d}\int_{\R^d}u(x)(w_\mu(y)+u(y))j_{m,\alpha}(|x-y|)dxdy\\
&=-\cE_{m,\alpha}(u,u)-\int_{\cU_\mu}\int_{\cU_\mu}u(x)(w_\mu(y)+u(y))j_{m,\alpha}(|x-y|)dydx\\
&\qquad -\int_{\cU_\mu}\int_{\R^d\setminus \cU_\mu}u(x)(w_\mu(y)+u(y))j_{m,\alpha}(|x-y|)dydx\\
&=-\cE_{m,\alpha}(u,u)-\int_{\cU_\mu}\int_{\cU_\mu}u(x)w_\mu^+(y)j_{m,\alpha}(|x-y|)dydx
+\int_{\cU_\mu}\int_{\cU_\mu}u(x)w_\mu(y)j_{m,\alpha}(|x-y^\mu|)dydx\\
&=-\cE_{m,\alpha}(u,u)
-\int_{\cU_\mu}\int_{\cU_\mu}u(x)(w_\mu^+(y)(j_{m,\alpha}(|x-y|)-j_{m,\alpha}(|x-y^\mu|))+w_\mu^-j_{m,\alpha}(|x-y^\mu|))dydx\\
&\le -\cE_{m,\alpha}(u,u).
\end{align*}
On the other hand, we have
\begin{align}
\label{eq:ineqweak}
\begin{split}
0&=\cE_{m,\alpha}(w_\mu,u)-\langle (\lambda-V)w_\mu, u\rangle - \langle (V-V^\mu)\varphi, u\rangle\\
&\le -\cE_{m,\alpha}(u,u)-\langle (\lambda-V)w_\mu, u\rangle- \langle (V-V^\mu)\varphi, u\rangle.
\end{split}
\end{align}
By Lemma \ref{lem:mono3} we know that ${\rm supp}(u)\subset \{x \in \R^d: \ x^\mu \not \in B_a\}=B_a(2\mu\mathbf{e}_1)$.
Hence
\begin{align*}
\label{eq:ineqweak2}
\begin{split}
\langle (\lambda-V)w_\mu, u\rangle&=\int_{\R^d}(\lambda-V(x))w_\mu(x)u(x)dx\\
&=-\int_{\R^d}(\lambda-V(x))u(x)^2dx=-\int_{B_a(2\mu\mathbf{e}_1)}(\lambda-V(x))u(x)^2dx
\ge -(\lambda+v)\Norm{u}{2}^2.
\end{split}
\end{align*}
Since the Dirichlet eigenvalues are translation-invariant, we also have
\begin{equation*}
\label{eq:ineqweak3}
\cE_{m,\alpha}(u,u) \ge \lambda_a \Norm{u}{2}^2.
\end{equation*}
Furthermore,
\begin{equation}
\label{eq:ineqweak4}
\langle (V-V^\mu)\varphi, u\rangle=\int_{B_a(2\mu\mathbf{e}_1)}(V(x)-V^\mu(x))\varphi(x)u(x)dx
=\int_{B_a(2\mu\mathbf{e}_1)}(V(x)+v)\varphi(x)u(x)dx\ge 0.
\end{equation}
Combining \eqref{eq:ineqweak}-\eqref{eq:ineqweak4} and using Lemma \ref{lem:mono4}, we obtain
\begin{equation*}
0 \le -(\lambda_a-(\lambda+v))\Norm{u}{2}^2 \le 0.
\end{equation*}
Thus necessarily $u \equiv 0$, which implies $w_\mu(x)\ge 0$ for every $x \in \cU_\mu$. Since $\mu \in (-a,0]$ is arbitrary,
we have that $w_\mu(x)\ge 0$ for all $\mu \in (-a,0]$ and every $x \in \cU_\mu$. Take any $0 \le r_1<r_2<a$ and let $x=
-r_1\mathbf{e}_1$, $y=-r_2\mathbf{e}_1$ and $\mu=-\frac{r_1+r_2}{2}$. Clearly, $y \in \cU_\mu$ and $x=y^\mu$. We have $w_\mu(y)
\ge 0$, which implies
\begin{equation*}
\chi(r_1)-\chi(r_2)=\varphi(x)-\varphi(y)=\varphi(y^\mu)-\varphi(y)\ge 0.
\end{equation*}
We then conclude the proof by making use of the continuity of $\varphi$.
\end{proof}

\section{Appendix}
\subsection{Denseness of $C_{\rm c}^\infty(\R^d)$ in $H^\Phi(\R^d)$}
Here we provide a constructive proof of Proposition \ref{thm:density}. First we show the following growth property of
the second moment of the jump measure over balls.
\begin{lemma}
\label{lem:Rsq}
For every $\Phi \in \cB_0$ we have
\begin{equation*}
\lim_{R \to \infty}\frac{1}{R^2}\int_{B_R}|x|^2j_\Phi(|x|)dx=0.
\end{equation*}
\end{lemma}
\begin{proof}
By using \eqref{jumpdens} we readily get
\begin{align*}
\frac{1}{R^2}\int_{B_R}|x|^2j_\Phi(|x|)dx&=\frac{\sigma_d}{R^2}\int_{0}^Rr^{d+1}j_\Phi(r)dr
=\frac{C_d}{R^2}\int_{0}^R\int_0^{\infty}r^{d+1}t^{-\frac{d}{2}}e^{-\frac{r^2}{4t}}\mu_\Phi(t)dtdr\\
&=\frac{C_d}{R^2}\int_0^{\infty}t^{-\frac{d}{2}}\mu_\Phi(t)\left(\int_{0}^Rr^{d+1}e^{-\frac{r^2}{4t}}dr\right)dt\\
& \stackrel {z=r^2/4t}{=}
\frac{C_d}{R^2}\int_0^{\infty}t\mu_\Phi(t)\gamma\left(\frac{d}{2}+1;\frac{R^2}{4t}\right)dt,
\end{align*}
where $\gamma(s;x):=\int_0^x z^{s-1}e^{-z}dz$ is the standard lower incomplete Gamma function. Writing $s=\frac{R^2}{4t}$,
we obtain
\begin{align*}
\frac{1}{R^2}\int_{B_R}|x|^2j_\Phi(|x|)dx&=C_dR^2\int_0^{\infty}s^{-3}\mu_\Phi\left(\frac{R^2}{4s}\right)
\gamma\left(\frac{d}{2}+1;s\right)ds\\
&=C_dR^2\left( \int_0^{1}+\int_1^{\infty}\right) s^{-3}\mu_\Phi\left(\frac{R^2}{4s}\right)\gamma\left(\frac{d}{2}+1;s\right)ds\\
&=:I_1(R)+I_2(R).
\end{align*}
To handle $I_1(R)$ we use that $\gamma\left(\frac{d}{2}+1;s\right)s^{-\frac{d}{2}-1} \to \frac{2}{d+2}$ as
$s \to 0$ and again make the substitution $s=\frac{R^2}{4t}$ to obtain
\begin{align*}
I_1(R)&\le C_dR^2\int_0^{1}s^{\frac{d}{2}-2}\mu_\Phi\left(\frac{R^2}{4s}\right)ds
= C_d\int_{R^2/4}^{\infty}R^d t^{-\frac{d}{2}}\mu_\Phi(t)dt
\le C_d\int_{R^2/4}^{\infty}\mu_\Phi(t)dt \to 0,
\end{align*}
due to $\int_{1}^{\infty}\mu(t)dt<\infty$.
	
Coming to $I_2(R)$, we use that $\gamma\left(\frac{d}{2}+1;s\right) \to \Gamma\left(\frac{d}{2}+1\right)$ as $s \to \infty$,
leading to
\begin{align*}
I_2(R)&\le C_dR^2\int_{1}^{\infty}s^{-3}\mu_\Phi\left(\frac{R^2}{4s}\right)ds
= \frac{C_d}{R^2}\int_0^{R^2/4}t\mu_\Phi(t)dt.
\end{align*}
Define $\overline{\mu}_\Phi(t):=\int_t^{\infty}\mu_\Phi(\tau)d\tau$. The function $\overline{\mu}_\Phi \in L^1(0,1)$ and is
non-increasing, hence $\lim_{t \to 0}t\overline{\mu}_\Phi(t)=0$ by the monotone density theorem (see, e.g., \cite[Th. 1.7.2]{B89}).
Integration by parts gives
\begin{align*}
I_2(R)& \leq
\frac{C_d}{R^2}\left(-\frac{R^2}{4}\overline{\mu}_\Phi\Big(\frac{R^2}{4}\Big)+\int_0^{R^2/4}\overline{\mu}_\Phi(t)dt\right)
=-\frac{C_d\overline{\mu}_\Phi(R^2/4)}{4}+\frac{C_d}{R^2}\int_0^{R^2/4}\overline{\mu}_\Phi(t)dt.
\end{align*}
Finally, again by the monotone density theorem and using that $\overline{\mu}_\Phi(t) \to 0$ as $t \to \infty$, we obtain
$\lim_{R \to \infty}\frac{C_d}{R^2}\int_0^{R^2/4}\overline{\mu}_\Phi(t)dt=0$,
which gives then $\lim_{R \to \infty}I_2(R)=0$.
\end{proof}

Now we are ready to prove that $C_{\rm c}^\infty(\R^d)$ is dense in $H^\Phi(\R^d)$.
Define $g_n:\R^d \to \R$, $n \in \N$, as a function in $C^\infty_{\rm c}(\R^d)$ such that $g_n(x)=1$ for every
$x \in B_n$, $g_n(x)=0$ for every $x \in \R^d\setminus B_{2n}$, $0 \le g_n(x)\le 1$ for every $x \in B_{2n}
\setminus B_n$ and $|\nabla g_n(x)|\le \frac{C}{n}$, with a constant $C$ independent of $n$. (This can be done by
applying a mollifier to the characteristic function of $B_n$, once we notice that ${\rm dist}(B_n,\partial B_{2n})
=n$ for all $n \in \N$). Furthermore, let $\sequ\varrho$ be a sequence of Friedrichs mollifiers of the form
$\varrho_n(x)=n^d\varrho(nx)$, $\supp\varrho(x)=B_1$. Fix $u \in H^{\Phi}(\R^d)$ and define $u_n:= \varrho_n \ast
(g_nu)$. Clearly, $u_n \in C_{\rm c}^\infty(\R^d)$, moreover,
	\begin{align*}
		\Norm{u_n-u}{2}^2&=\int_{\R^d}\left|\int_{\R^d}\varrho_n(x-y)g_n(x)u(x)dx-u(y)\right|^2dy\\
		&\le 2\int_{\R^d}\left|\int_{\R^d}\varrho_n(x-y)g_n(x)u(x)dx-\int_{\R^d}\varrho_n(x-y)u(x)dx\right|^2dy\\
		&\quad +2\int_{\R^d}\left|\int_{\R^d}\varrho_n(x-y)u(x)dx-u(y)\right|^2dy
		=:2(I^{(1)}_n+I^{(2)}_n).
	\end{align*}
	It is well-known that $I^{(2)}_n:=\Norm{\varrho_n \ast u-u}{2} \to 0$ as $n \to \infty$. Furthermore, by the
generalized Minkowski inequality (see \cite[ineq. 202]{HLP52}),
	\begin{equation*}
		I^{(1)}_n \le \Norm{g_nu-u}{2}^2=\int_{B_{n+1}\setminus B_n}(1-g_n(x))^2|u(x)|^2dx
\le \int_{B_{n+1}\setminus B_n}|u(x)|^2dx \to 0.
	\end{equation*}
	Hence $u_n \to u$ in $L^2(\R^d)$. Let now $\overline{u}_n=g_nu$. We have
	\begin{align*}
		[u_n-u]^2_{\Phi}&=\int_{\R^d}\int_{\R^d}|u_n(y)-u(y)-u_n(x)+u(x)|^2j_\Phi(|x-y|)dxdy\\
		&=\int_{\R^d}\int_{\R^d}|u_n(y)-\overline{u}_n(y)+\overline{u}_n(y)-u(y)-u_n(x)+\overline{u}_n(x)
-\overline{u}_n(x)+u(x)|^2j_\Phi(|x-y|)dxdy\\
		&\le 2[u_n-\overline{u}_n]^2_{\Phi}+2[\overline{u}_n-u]_{\Phi}^2.
	\end{align*}
First consider the second semi-norm. Define $\overline{g}_n=1-g_n$ such that
	\begin{align*}
		[\overline{u}_n-u]_{\Phi}^2&=\int_{\R^d}\int_{\R^d}|\overline{g}_n(x)u(x)-\overline{g}_n(y)u(y)|^2j_\Phi(|x-y|)dxdy\\
		&=2\int_{B_n}\int_{B_{2n}\setminus B_n}|\overline{g}_n(y)u(y)|^2j_\Phi(|x-y|)dxdy\\
		&\quad +2\int_{B_n}\int_{\R^d\setminus B_{2n}}|u(y)|^2j_\Phi(|x-y|)dxdy\\
		&\quad +\int_{B_{2n}\setminus B_n}\int_{B_{2n}\setminus B_n}|\overline{g}_n(x)u(x)-\overline{g}_n(y)u(y)|^2j_\Phi(|x-y|)dxdy\\
		&\quad +2\int_{B_{2n}\setminus B_n}\int_{\R^d\setminus B_{2n}}|\overline{g}_n(x)u(x)-u(y)|^2j_\Phi(|x-y|)dxdy\\
		&\quad +\int_{\R^d\setminus B_{2n}}\int_{\R^d\setminus B_{2n}}|u(x)-u(y)|^2j_\Phi(|x-y|)dxdy\\
		&=2I^{(3)}_n+2I^{(4)}_n+I^{(5)}_n+2I^{(6)}_n+I^{(7)}_n.
	\end{align*}
Since for every $x \in B_n$ we have that $\overline{g}_n(x)=0$ and $\Norm{\nabla \overline{g}_n}{\infty}\le \frac{C}{n}$,
we get
	\begin{equation*}
		I^{(3)}_n \le \frac{C}{n^2}\int_{B_n}\int_{B_{2n}\setminus B_n}|u(y)|^2|x-y|^2j_\Phi(|x-y|)dxdy
\leq \frac{C\Norm{u}{2}}{n^2}\int_{B_{3n}}|x|^2j_\Phi(|x|)dx.
	\end{equation*}
The second bound follows by the observation that since $y \in B_{2n}\setminus B_n$, we have $B_n \subset B_{3n}(y)$.
	Taking the limit as $n \to \infty$, by Lemma \ref{lem:Rsq} we have $\lim_{n \to \infty}I^{(3)}_n =0$.

For the next integral we simply estimate
	\begin{equation*}
		I^{(4)}_n \le \Norm{u}{2}\int_{\R^d \setminus B_{2n}}j_\Phi(|x|)dx,
	\end{equation*}
which on taking the limit gives $\lim_{n \to \infty}I^{(4)}_n=0$, given that
$\int_{\R^d \setminus B_1}j_\Phi(|x|)dx<\infty$.

Next consider $I^{(5)}_n$. Adding and subtracting $\overline{g}_n(y)u(x)$ we have
	\begin{align*}
		I^{(5)}_n &\le 2\int_{B_{2n}\setminus B_n}\int_{B_{2n}\setminus B_n}
|\overline{g}_n(x)-\overline{g}_n(y)|^2|u(x)|^2j_{\Phi(|x-y|)}dxdy\\
		&\qquad +2\int_{B_{2n}\setminus B_n}\int_{B_{2n}\setminus B_n}|\overline{g}_n(x)|^2|u(x)-u(y)|^2j_{\Phi(|x-y|)}dxdy\\
		&\le \frac{2C}{n^2}\int_{B_{2n}\setminus B_n}\int_{B_{2n}\setminus B_n}|x-y|^2|u(x)|^2j_{\Phi(|x-y|)}dxdy\\
		&\qquad +2\int_{B_{2n}\setminus B_n}\int_{B_{2n}\setminus B_n}|u(x)-u(y)|^2j_{\Phi(|x-y|)}dxdy
		=2(I_n^{(8)}+I_n^{(9)}).
	\end{align*}
Since $x \in B_{2n}$, we have that $B_{2n}\subset B_{4n}(x)$ and thus
\begin{equation*}
I_n^{(8)} \le \frac{2C\Norm{u}{2}}{n^2}\int_{B_{4n}}|y|^2j_\Phi(|y|)dy \to 0,
\end{equation*}
again by Lemma \ref{lem:Rsq}. On the other hand, $\lim_{n \to \infty}I^{(9)}_n=0$ by dominated convergence, since
$[u]_{\Phi}^2<\infty$. Hence we have $\lim_{n \to \infty}I^{(5)}_n=0$.

To estimate $I^{(6)}_n$, we add and subtract $u(x)=\bar{H}_n(y)u(x)$ using that $y \in \R^d \setminus B_{2n}$, and
further split off the integral giving
\begin{align*}
I^{(6)}_n &\le 2\int_{B_{2n}\setminus B_n}\int_{B_{3n}\setminus B_{2n}}|\overline{g}_n(x)-\overline{g}_n(y)|^2|u(x)|^2
j_{\Phi}(|x-y|)dxdy\\
&\qquad
+2\int_{B_{2n}\setminus B_n}\int_{\R^d\setminus B_{3n}}|\overline{g}_n(x)-\overline{g}_n(y)|^2|u(x)|^2j_{\Phi}(|x-y|)dxdy\\
&\qquad +2\int_{B_{2n}\setminus B_n}\int_{\R^d\setminus B_{2n}}|\overline{g}_n(x)|^2|u(x)-u(y)|^2j_{\Phi}(|x-y|)dxdy
=2(I^{(10)}_n+I^{(11)}_n+I^{(12)}_n).
\end{align*}
Concerning $I_n^{(10)}$, using that $x \in B_{3n} \setminus B_{2n}$ and so $B_{2n}\subset B_{5n}(x)$, we have,
\begin{equation*}
	I_n^{(10)} \le \frac{C\Norm{u}{2}}{n^2}\int_{B_{5n}}|y|^2j_{\Phi}(|y|)dy \to 0,
\end{equation*}
by Lemma \ref{lem:Rsq}. On the other hand, if $x \in \R^d\setminus B_{3n}$ and $y \in B_{2n}\setminus B_n$, we have
$|x-y| \ge 1$ and thus
\begin{equation*}
I^{(11)}_n \le 8\left(\int_{\R^d \setminus B_{3n}}|u(x)|^2dx\right)\left(\int_{\R^d \setminus B_1}j_\Phi(|y|)dy\right)
\to 0
\end{equation*}
since $u \in L^2(\R^d)$. Furthermore, $I_n^{(12)} \to 0$ by dominated convergence, as $[u]^2_{\Phi}<\infty$. Finally,
$I^{(7)}_n \to 0$ simply by the dominated convergence theorem. Combining the limits $I^{(i)}_n \to 0$ for $i=3,\dots,7$,
we have $\lim_{n \to \infty}[\overline{u}_n-u]^2_{\Phi}=0$. Furthermore	$[\overline{u}_n]_{\Phi}\le [u]_{\Phi}+
[\overline{u}_n-u]_{\Phi}$ and so $\limsup_{n \to \infty}[u_n]_{\Phi}\le [u]_{\Phi}<\infty$, which leads to
$[\overline{u}_n] \le C$ for a constant $C>0$ independent of $n$.

Note that by the Jensen inequality
\begin{align*}
[u_n-\overline{u}_n]^2_{\Phi}
&=\int_{\R^d}\int_{\R^d}|u_n(y)-\overline{u}_n(y)-u_n(x)+\overline{u}_n(x)|^2j_\Phi(|x-y|)dxdy\\
&=\int_{\R^d}\int_{\R^d}\left|\int_{\R^d}\varrho_n(z)\overline{u}_n(y-z)dz-\overline{u}_n(y)
-\int_{\R^d}\varrho_n(z)\overline{u}_n(x-z)dz+\overline{u}_n(x)\right|^2j_\Phi(|x-y|)dxdy\\
&=\int_{\R^d}\int_{\R^d}\left|\int_{\R^d}\varrho_n(z)(\overline{u}_n(y-z)-\overline{u}_n(y)
-\overline{u}_n(x-z)+\overline{u}_n(x))dz\right|^2j_\Phi(|x-y|)dxdy\\
&\le \int_{\R^d}\int_{\R^d}\int_{\R^d}\varrho_n(z)|\overline{u}_n(y-z)-\overline{u}_n(y)-\overline{u}_n(x-z)
+\overline{u}_n(x)|^2j_\Phi(|x-y|)dzdxdy\\
&\le 3\int_{\R^d}\int_{\R^d}\int_{\R^d}\varrho_n(z)(|\overline{u}_n(y-z)-u(y-z)-\overline{u}_n(x-z)+u(x-z)|^2\\
&\qquad +|u(y-z)-u(x-z)-u(y)+u(x)|^2+|\overline{u}_n(y)-u(y)-\overline{u}_n(x)+u(x)|^2)j_\Phi(|x-y|)dzdxdy\\
&\le 6[\overline{u}_n-u]^2_{\Phi}+3\int_{\R^d}\varrho_n(z)\int_{\R^d \times \R^d}|(u(y-z)-u(x-z)-u(y)+u(x))
\sqrt{j_\Phi(|x-y|)}|^2.
\end{align*}
Denote $G(x,y)=(u(y)-u(x))\sqrt{j_{\Phi}(|x-y|)} \in L^2(\R^d \times \R^d)$. By the above we then have
\begin{align*}
[u_n-\overline{u}_n]^2_{\Phi}&\le 6[\overline{u}_n-u]_{\Phi}^2+3\int_{\R^d}\Norm{G(\cdot-z,\cdot-z)-G}
{L^2(\R^d \times \R^d)}^2\varrho_n(z)dz\\
&=6[\overline{u}_n-u]_{\Phi}^2+3n^d\int_{\R^d}\Norm{G(\cdot-z,\cdot-z)-G}{L^2(\R^d \times \R^d)}^2\varrho(nz)dz\\
&=6[\overline{u}_n-u]_{\Phi}^2+3\int_{\R^d}\Norm{G\left(\cdot-\frac{z}{n},\cdot-\frac{z}{n}\right)
-G}{L^2(\R^d \times \R^d)}^2\varrho(z)dz\\
&\le 6[\overline{u}_n-u]_{\Phi}^2
+3\Norm{\varrho}{\infty}\int_{B_1}\Norm{G\left(\cdot-\frac{z}{n},\cdot-\frac{z}{n}\right)-G}
{L^2(\R^d \times \R^d)}^2dz.
\end{align*}
Since $G \in L^2(\R^d \times \R^d)$, the function $z \in \R^d \mapsto
\Norm{G\left(\cdot-z,\cdot-z\right)-G}{L^2(\R^d \times \R^d)}^2 \in \R$ is continuous and thus
$\sup_{z \in B_1}\Norm{G\left(\cdot-z,\cdot-z\right)-G}{L^2(\R^d \times \R^d)}^2<\infty$.
Hence we can use the dominated convergence theorem to obtain
\begin{equation*}\label{eq:firstsemin}
[u_n-\overline{u}_n]^2_{\Phi} \le 6[\overline{u}_n-u]_{\Phi}^2+3\Norm{\varrho}{\infty}
\int_{B_1}\Norm{G\left(\cdot-\frac{z}{n},\cdot-\frac{z}{n}\right)-G}{L^2(\R^d \times \R^d)}^2dz \to 0
\end{equation*}
as $n \to \infty$.

We note that the strategy of the proof above follows \cite[Th. 3.1]{BGV21}. In particular, in this proof we actually
obtained a truncation result in the spirit of \cite[Appx. B]{BGV21}.

\subsection{Equivalent expression}
We prove relation \eqref{moreequiv}.
First of all, since $\cS(\R^d)\subset C^2(\R^d)\cap L^\infty(\R^d)$, Lemma \ref{lem:existence1} ensures that $\Phi(-\Delta)u$
is well-defined. To see that we can exchange the order of Fourier transform and the integral in the definition of
$\Phi(-\Delta)u$, first we show that
\begin{equation*}
\int_{\R^d}\int_{\R^d}|u(x+h)-2u(x)+u(x-h)|j_{\Phi}(|h|)dhdx<\infty.
\end{equation*}
Let $D^2u$ be the Hessian matrix of $u$ and denote $|D^2u(x)|=\sup_{h \in \R^d \setminus \{0\}}\frac{|D^2u(x)h|}{|h|}$.
We split off the integral like
\begin{align*}
\int_{\R^d}\int_{\R^d}|u(x+h)-2u(x)+u(x-h)|j_{\Phi}(|h|)dhdx&=\int_{\R^d}\int_{B_1}|u(x+h)-2u(x)+u(x-h)|j_{\Phi}(|h|)dhdx\\
&\quad +\int_{\R^d}\int_{\R^d \setminus B_1}|u(x+h)-2u(x)+u(x-h)|j_{\Phi}(|h|)dhdx\\
&=I_1+I_2.
\end{align*}
Note that there exists a function $\theta: (x,h) \in \R^d \times \R^d \mapsto \theta(x,h) \in B_1(x)$ such that
\begin{align*}
I_1 &\le \int_{\R^d}\int_{B_1}|D^2u(\theta(x,h))||h|^2j_{\Phi}(|h|)dhdx\\
&=\int_{B_2}\int_{B_1}|D^2u(\theta(x,h))||h|^2j_{\Phi}(|h|)dhdx
 + \int_{\R^d \setminus B_2}\int_{B_1}|D^2u(\theta(x,h))||h|^2j_{\Phi}(|h|)dhdx\\
&=I_3+I_4.
\end{align*}
Concerning $I_3$, observe that if $x \in B_2$, then $\theta(x,h) \in B_3$ for every $h \in B_1$. Hence we have
\begin{equation*}
I_3 \le 3^d\omega_d\Norm{|D^2u|}{L^\infty(B_3)}\int_{B_1}|h|^2j_{\Phi}(|h|)dh<\infty,
\end{equation*}
where $\omega_d$ is the Lebesgue measure of $B_1$. On the other hand, if $x \in \R^d \setminus B_2$, then $\theta(x,h)
\in \R^d \setminus B_{|x|-1}$ and then we have
\begin{equation*}
I_4 \le 3^d\omega_d\left(\int_{\R^d \setminus B_2}\Norm{|D^2u|}{L^\infty(\R^d \setminus B_{|x|-1})}dx\right)
\left(\int_{B_1}|h|^2j_{\Phi}(|h|)dh\right).
\end{equation*}
However, since $u \in \cS(\R^d)$, it is clear that
\begin{equation*}
|z|^{2}|D^2u(z)|\le d\sup_{z \in \R^d \setminus B_1}|z|^2
\max_{1\le i,j \le d}\left|\frac{\partial^2 \, u}{\partial x_i \, \partial x_j}(z)\right|
=d\max_{1\le i,j \le d}\sup_{z \in \R^d \setminus B_1}|z|^2\left|\frac{\partial^2 \, u}{\partial x_i \,
\partial x_j}(z)\right|=:C_{\cS(\R^d)}(u)<\infty
\end{equation*}
for all $z \in \R^d \setminus B_{|x|-1}$, which leads to
\begin{equation*}
|D^2u(z)|\le \frac{C_{\cS}(u)}{|z|^2}\le \frac{C_{\cS}(u)}{(|x|-1)^2}, \quad z \in \R^d \setminus B_{|x|-1}.
\end{equation*}
On taking the supremum on $z$ we obtain
\begin{equation*}
\Norm{|D^2u|}{L^\infty(\R^d \setminus B_{|x|-1})}\le \frac{C_{\cS(\R^d)}(u)}{(|x|-1)^2}.
\end{equation*}
In particular,
\begin{equation*}
\int_{\R^d \setminus B_2}\Norm{|D^2u|}{L^\infty(\R^d \setminus B_{|x|-1})}dx
\le C_{\cS(\R^d)}(u)\int_{\R^d \setminus B_{2}}\frac{dx}{(|x|-1)^{2}}<\infty,
\end{equation*}
implying $I_4<\infty$. Finally, considering $I_2$ we have
\begin{equation*}
I_2 \le 4\Norm{u}{\infty}\int_{\R^d \setminus B_1}j_\Phi(|h|)dh<\infty.
\end{equation*}
Hence we can use Fubini's theorem and the properties of the Fourier transform to obtain
\begin{equation*}
\cF[\Phi(-\Delta)u](\xi)=-\frac{\widehat{u}(\xi)}{2}\int_{\R^d}(e^{i(h\cdot\xi)}+e^{-i(h\cdot\xi)}-2)
j_{\Phi}(|h|)dh=\widehat{u}(\xi)\int_{\R^d}(1-\cos(h \cdot \xi))j_{\Phi}(|h|)dh.
\end{equation*}
By a similar argument as in Lemma \ref{lem:GagliardotoFourier} and also using Lemma \cite[Lem. 2.1]{JS05},
it follows then that
\begin{equation*}
\int_{\R^d}(1-\cos(h \cdot \xi))j_{\Phi}(|h|)dh=f(|\xi|^2),
\end{equation*}
which completes the proof.


\begin{thebibliography}{00}
\bibitem{AL89}
F.J. Almgren, E.H. Lieb:
Symmetric decreasing rearrangement is sometimes continuous, \emph{J. Amer. Math. Soc.} \textbf{2},
683-773, 1989

\bibitem{AL}
G. Ascione, J. L\H{o}rinczi:
Potentials for non-local Schr\"{o}dinger operators with zero eigenvalues,
\textit{J. Differ. Equations} \textbf{317}, 264-364, 2022

\bibitem{AL2}
G. Ascione, J. L\H{o}rinczi:
Bulk behaviour of ground states for relativistic Schr\" odinger operators with compactly supported
potentials, arXiv:2107.11580, 2021

\bibitem{B89}
N.H. Bingham, C.M. Goldie, J.L. Teugels:
\textit{Regular Variation}, Cambridge University Press, 1989

\bibitem{BL19}
A. Biswas, J. L\H{o}rinczi:
Maximum principles and Aleksandrov-Bakelman-Pucci type estimates for non-local Schr\"{o}dinger equations
with exterior conditions, \emph{SIAM J. Math. Anal.} \textbf{51}, 1543--1581, 2019

\bibitem{BL21}
A. Biswas, J. L\H{o}rinczi:
Hopf's lemma for viscosity solutions to a class of non-local equations with applications, \emph{Nonlinear Anal. Theory
Methods Appl.} \textbf{204}, 112194, 2021


\bibitem{BGV21}
L. Brasco, D. G\'omez-Castro, J. L. V\'azquez:
Characterisation of homogeneous fractional Sobolev spaces, \textit{Calc. Var. Partial Differ. Equ.} \textbf{60.2}, 1-40,
2021

\bibitem{BLP14}
L. Brasco, E. Lindgren, E. Parini:
The fractional Cheeger problem, \textit{Interfaces Free Boundaries} \textbf{16}, 419-458, 2014

\bibitem{BPS16}
L. Brasco, E. Parini, M. Squassina:
Stability of variational eigenvalues for the fractional $p$-Laplacian. \textit{Discrete Contin. Dyn. Syst. Ser. A}
\textbf{36}: 1813, 2016

\bibitem{B11}
H. Br\'ezis: \textit{Functional Analysis, Sobolev Spaces and Partial Differential Equations}, Springer, 2011

\bibitem{CS06}
Z.Q. Chen, R. Song:
Continuity of eigenvalues of subordinate processes in domains, \emph{Math. Z.} \textbf{252},
71-89, 2006

\bibitem{CHL17}
T. Cheng, G. Huang, C. Li:
The maximum principles for fractional Laplacian equations and their applications,
\textit{Comm. Contemp. Math.} \textbf{19}, 1750018, 2017

\bibitem{D14}
B. Dacorogna:
\textit{Introduction to the Calculus of Variations}, World Scientific, 2014

\bibitem{DM12}
G. Dal Maso:
\textit{An Introduction to $\Gamma$-Convergence}, Springer, 2012

\bibitem{DH74}
B. DeFacio, C.L. Hammer:
Remarks on the Klauder phenomenon, \textit{J. Math. Phys.} \textbf{15}, 1071-1077, 1974

\bibitem{DA19}
G.F. Dell'Antonio:
Contact interactions and Gamma convergence, \emph{Front. Phys.}, Sec. Statistical and Computational Physics,
09 April 2019

\bibitem{DA21}
G.F. Dell'Antonio:
Contact interactions and Gamma convergence: new tools in quantum mechanics, \emph{Eur. Phys. J. Plus}
\textbf{136:392}, 2021

\bibitem{DLN}
C.S. Deng, W. Liu, E. Nane:
Finite time blowup of solutions to SPDEs with Bernstein functions of the Laplacian,
\emph{ Potential Anal.} 2022, https://doi.org/10.1007/s11118-021-09978-1


\bibitem{DPV12}
E. Di Nezza, G. Palatucci, E. Valdinoci:
Hitchhiker's guide to the fractional Sobolev spaces, \textit{Bull. Sci. Math.} \textbf{136},
521-573, 2012

\bibitem{DM72}
H. Dym, H.P. McKean:
\textit{Fourier Series and Integrals}, Academic Press, 1972

\bibitem{FJ15}
M.M. Fall, S. Jarohs:
Overdetermined problems with fractional Laplacian, \textit{ESAIM Control Optim. Calc. Var.} \textbf{21},
924-938, 2015

\bibitem{HLP52}
G.H. Hardy, J.E. Littlewood, G. P\'olya:
\textit{Inequalities}, Cambridge University Press, 1952

\bibitem{HIL12}
F. Hiroshima, T. Ichinose, J. L\H{o}rinczi:
Path integral representation for Schr\"odinger operators with Bernstein functions of the Laplacian,
\emph{Rev. Math. Phys.} \textbf{24} (2012), 1250013

\bibitem{IW20}
A. Ishida, K. Wada:
Threshold between short and long-range potentials for non-local Schr\"odinger operators,
\emph{J. Math. Phys. Anal. Geom.} \textbf{23:32}, 2020

\bibitem{ILS}
A. Ishida, J. L\H orinczi, I. Sasaki:
Absence of embedded eigenvalues for non-local Schr\"o\-din\-ger operators, \emph{J. Evol. Eq.}
\textbf{22:82}, 1-30, 2022

\bibitem{JS05}
N. Jacob, R.L. Schilling:
Function spaces as Dirichlet spaces (about a paper by Maz'ya and Nagel),
\textit{Z. Anal. Anwend.} \textbf{24}, 3-28, 2005

\bibitem{KL17}
K. Kaleta, J. L\H{o}rinczi:
Fall-off of eigenfunctions for non-local Schr\"odinger operators with decaying potentials,
\emph{Potential Anal.} \textbf{46} (2017), 647-688

\bibitem{K80}
T. Kato:
\emph{Perturbation Theory for Linear Operators}, 2nd edition, Springer, 1980

\bibitem{KSV12}
P. Kim, R. Song, Z. Vondra\v{c}ek:
Potential theory of subordinate Brownian motions revisited, in: \textit{Stochastic Analysis and
Applications to Finance: Essays in Honour of Jia-an Yan} (T. Zhang, X. Zhou, eds.), Interdisciplinary
Mathematical Sciences: Vol. 13, World Scientific, pp 243-290, 2012

\bibitem{KM}
M. Kwa\'snicki, J. Mucha:
Extension technique for complete Bernstein functions of the Laplace operator, \emph{J. Evol. Equ.} \textbf{18} (2018),
1341--1379

\bibitem{LUY}
L. Li, T. Uemura, J. Ying:
Weak convergence of regular Dirichlet subspaces, \emph{Osaka J. Math.} \textbf{54}, 435-455, 2017

\bibitem{LHB}
J. L\H{o}rinczi, F. Hiroshima, V. Betz:
\emph{Feynman-Kac-Type Theorems and Gibbs Measures on Path Space. With Applications to Rigorous Quantum
Field Theory}, de Gruyter Studies in Mathematics \textbf{34}, Walter de Gruyter, 2011; 2nd rev. exp. ed.,
vol. 1, 2020

\bibitem{MN78}
W. Maz'ya, J. Nagel:
\"{U}ber \"aquivalente Normierung der anisotropen Funktionalr\"aume $H^\mu(\R^n)$, \textit{Beitr\"age Anal.}
\textbf{12}, 7-17, 1978

\bibitem{M00}
W. McLean:
\textit{Strongly Elliptic Systems and Boundary Integral Equations}, Cambridge University Press, 2000

\bibitem{M93}
U. Mosco:
Composite media and asymptotic Dirichlet forms, \emph{J. Funct. Anal.} \textbf{123}, 368-421, 1994

\bibitem{P04}
A.C. Ponce:
A new approach to Sobolev spaces and connections to $\Gamma$-convergence, \emph{Calc. Var.} \textbf{19}, 229-255, 2004

\bibitem{R02}
M. Ryznar:
Estimates of Green function for relativistic $\alpha$-stable process, \emph{ Potential Anal.} \textbf{17} (2002),
1-23

\bibitem{SSV}
R. Schilling, R. Song, Z. Vondra\v{c}ek:
\emph{Bernstein Functions}, Walter de Gruyter, 2010

\bibitem{S73}
B. Simon: Quadratic forms and Klauder's phenomenon: a remark on very singular perturbations,
\textit{J. Funct. Anal.} \textbf{14}, 295-298, 1973

\bibitem{SL16}
X. Song, L. Li:
Regular Dirichlet subspaces and Mosco convergence, \emph{Chinese Ann. Math. Ser. A} \textbf{37},
1-14, 2016

\bibitem{T09}
G. Teschl:
\textit{Mathematical Methods in Quantum Mechanics}, Graduate Studies in Mathematics 99, 2009

\bibitem{Wat}
G.N. Watson:
\emph{A Treatise on the Theory of Bessel Functions}, Cambridge University Press, 2nd ed., 1966

\bibitem{W80}
J. Weidmann:
Continuity of the eigenvalues of self-adjoint operators with respect to the strong operator topology,
\textit{Integr. Equ. Oper. Theory} \textbf{3}, 138-142, 1980
\end{thebibliography}
\end{document}